\documentclass[12pt]{article}
\usepackage{amsmath, amsthm, amssymb,amsfonts,amscd}
\usepackage{graphicx,epsfig,latexsym}
\usepackage{calc,xcolor,subfigmat,mathrsfs,dsfont}
\usepackage{graphicx,epsfig}
\usepackage{array}
\usepackage{amsfonts}
\usepackage{amscd}
\usepackage{enumitem}
\usepackage{mathrsfs}
\usepackage{latexsym}
\usepackage{lineno}
\usepackage{datetime2}

\def\nM{\mathsf{M}}
\def\nN{\mathsf{N}}
\def\nH{\mathsf{H}}

\newcommand{\conn}[1]{\nabla_{#1}}
\newcommand{\cann}[1]{\Delta_{#1}}
\newcommand{\Sl}{\mathfrak{sl}_2}

\newcommand{\subs}[1]{#1_\star}
\newcommand{\mult}[1]{#1}

\newcommand{{\E}}{{\mathsf E}}

\newcommand{\im}{{\rm im\ }}
\newcommand{\0}{\mathbf{0}}

\newcommand{\nv}{\mathsf{v}}
\newcommand{\nw}{\mathsf{w}}

\newtheorem{thm}{Theorem}[section]
\newtheorem{cor}[thm]{Corollary}
\newtheorem{lem}[thm]{Lemma}

\theoremstyle{definition}
\newtheorem{defn}[thm]{Definition}

\newtheorem{rem}[thm]{Remark}

\newtheorem{remark}[thm]{Remark}

\numberwithin{equation}{section}
\usepackage{color}
\newcounter{IssueCounter}
\newtheorem{Issue}[IssueCounter]{Issue}
\newcommand{\issue}[2]{
\begin{Issue}[\textcolor{red}{#1}]{\textcolor{blue}{#2}}\end{Issue}}
\def\be {\begin{equation}}
\def\ee {\end{equation}}
\def\ba {\begin{eqnarray}}
\def\ea {\end{eqnarray}}
\def\bpr {\begin{proof}}
\def\epr {\end{proof}}
\def\bes {\begin{equation*}}
\def\ees {\end{equation*}}
\def\bas {\begin{eqnarray*}}
\def\eas {\end{eqnarray*}}

\begin{document}
\renewcommand {\thefootnote}{\dag}
\renewcommand {\thefootnote}{\ddag}
\renewcommand {\thefootnote}{ }
\title{Normal form for maps with nilpotent linear part}
\author{
Fahimeh Mokhtari\footnote{Email Fahimeh Mokhtari: fahimeh.mokhtari.fm@gmail.com}\\
Ernst Roell\footnote{Email Ernst R\"oell: ernstroell@gmail.com}\\
Jan A. Sanders \footnote{Email Jan Sanders: jan.sanders.a@gmail.com}
}
\maketitle
\begin{abstract}
The normal form for an \(n\)-dimensional map with irreducible nilpotent linear part is determined using \(\Sl\)-representation theory.
We sketch by example how the reducible case can also be treated in an algorithmic manner.
The construction (and proof) of the \(\Sl\)-triple from the nilpotent linear part is more complicated than one would hope for,
but once the abstract \(\Sl\) theory is in place, both the description of the normal form and the computational splitting to compute the generator of the coordinate transformation
can be handled explicitly in terms of the nilpotent linear part without the explicit knowledge of the triple.
If one wishes one can compute the normal form such that it is guaranteed to lie in the kernel of an operator and one can be sure that this is really a
normal form with respect to the nilpotent linear part; one can state that the normal form is in \(\Sl\)-style.

Although at first sight the normal form theory for maps is more complicated than for vector fields in the nilpotent case,
it turns out that the final result is much better. Where in the vector field case one runs into invariant theoretical problems when the dimension gets larger if
one wants to describe the general form of the normal form,
for maps we obtain results without any restrictions on the dimension.

In the literature only the 2-dimensional nilpotent case has been described sofar, as far as we know. 
\end{abstract}
\section{Introduction}
Normal form theory (cf. \cite[Appendix A]{sanders2007averaging} for historical remarks) aims to transform a given system
to some canonical form, called the normal form, in order to analyse the
possible bifurcations when the parameters of the system pass a critical value or to compute approximate orbits. 
If one is lucky, the normal form is simpler or has more symmetry than the original system, but this is not part of the definition.

For vector fields this theory is well
developed, see for instance \cite{kuznetsov2004elements,sanders2007averaging},
and many different viewpoints have been developed over the years, invoking many
different branches of mathematics, such as Lie theory, cohomology theory and
representation theory,
cf. \cite{sanders2007averaging}.

However, for discrete dynamical systems the development of 
normal form theory has received little attention in the literature. 
\begin{rem}
In the literature, maps are often seen as Poincar\'e maps for some vector field and 
and if the linearization is of the type identity plus nilpotent matrix, it is called nilpotent,
since it is related  to the Poincar\'e map of a vector field with nilpotent linear part.
This is not our definition of nilpotent.
\end{rem}
There are few  paper  regarding the study of normal forms for maps,
\cite{gramchev2005normal,MR1735239,wang2008further},
and there seems to be no general and systematic approach to the construction of
normal forms of maps.
However, the recent appearance of \cite{kuznetsov2019numerical}, confirms that
there is indeed a need for the theoretical development of normal form theory
for maps.

One of the main differences  between maps and vector fields lies
in fact that the so-called homological operator is not linear with respect to
the linear part of the map. This plays a role when we want to lift decompositions, like the additive semisimple-nilpotent decomposition,
or when we want to compute Lie brackets.
For instance, for vector fields the additive semisimple-nilpotent splitting for the linear part of the
vector field induces an additive semisimple-nilpotent splitting for the homological operator. 
In the case of maps, the additive semisimple-nilpotent splitting of the linear part does
{\em not} induce an additive semisimple-nilpotent splitting of the homological operator, due to
the lack of linearity.  The multiplicative semisimple-nilpotent splitting does not have this problem and
might be the more natural splitting to consider. But here nilpotent means identity plus nilpotent, so it will not be a trivial application of the theory in this paper.

Similarily, for vector fields with nilpotent linear part, embedding the linear
part into an 
$\mathfrak{sl}_{2}$-triple induces an 
$\mathfrak{sl}_{2}$-triple in the homological operator. 

Nilpotent normal forms for vector fields have been  studied in the
literature and over the years it has become clear that there is intricate
interplay between nilpotent normal forms and the finite dimensional representations of 
$\mathfrak{sl}_{2}$, see
\cite{MR855083,cushman1988normal,cushman1988splitting,gazor2013normal,gazor2013volume,gazor2014normal}.
The representation theory of 
$\Sl$ has also played a role in the normalisation of
nilpotent vector fields depending on parameters. 
For instance, \cite{mokhtari2019versal} has constructed the nilpotent normal form of versal
deformations of nilpotent and nonsemisimple vector fields. 
\\
The construction of the normal form for maps with nilpotent linear part has
received little attention in the literature.   The only  case that we are aware of   is  the two dimensional case, see   \cite{MR1735239}.

In this paper, we construct from the nilpotent linear part an  $\Sl$-action that  can be used to compute the normal form and we give an explicit expression for the normal form in the irreducible case, followed by several reducible cases, which tend to be slightly more complicated to analyze.

The Jacobson-Morozov theorem \cite[Section 12.5]{sanders2007averaging,mokhtari2019versal} guarantees that any nilpotent element of a reductive Lie
algebra can be embedded in an
$\Sl$-triple, which   
consists of three elements
$\langle N,H,M \rangle$ that satisfy the commutator
relations 
\begin{eqnarray*}
[M,N] &=& H,\quad {[}H,N] = -2 N,\quad
{[}H,M] = 2 M.
\end{eqnarray*}
This is isomorphic to the Lie algebra 
$\Sl$, which is spanned by the matrices
\[
\mathbf{n}
= 
\begin{pmatrix}
0 & 1 \\
0 & 0 
\end{pmatrix}
,
\quad 
\mathbf{m}
=
\begin{pmatrix}
0 & 0 \\
1 & 0
\end{pmatrix}
,
\quad
\mathbf{h}
=
\begin{pmatrix}
-1 & 0 \\
0 &1
\end{pmatrix}
.
\]
For further details on the study of Lie algebras and the finite dimensional
representations of 
$\mathfrak{sl}_{2}$, we refer the reader to
\cite{humphreys2012introduction}.
The idea of normal form theory is to transform a map 
$f$ through a conjugation by a near identity transformation to a new map 
$\bar{f}$, in such a way that if we were to repeat the process, \(\bar{f}\) would prove to be stable under it. 
This gives rise to the so-called homological operator (see Definition \ref{def:HomologicalOperator}), which plays a central
role in normal form theory. 
In analogy to normal form theory for vector fields the aim is to remove the
terms in the image of the homological operator. 
Equivalently, the transformed map 
$\bar{f}$ is in normal form when its nonlinear part lies in a complement to the image of the homological
operator. 
This is where the finite dimensional representations of
$\mathfrak{sl}_{2}$ come into play. 

For an 
$\Sl$-triple 
$\langle N,H,M\rangle$ acting on a finite dimensional vector space, the space
decomposes as 
$\im N \oplus \ker M$. 
Therefore, embedding the homological operator into an
$\Sl$-triple provides an explicit
description of the normal form {\em style} (as the kernel of the operator \(M\)) of a map with nilpotent linear part.
Moreover, one can in general compute \(\ker M\) and the \(H\)-eigenvalues of the elements in the kernel.
The dimension of the irreducible representation generated as an \(N\)-orbit of the kernel element is its eigenvalue plus one.
This allows one to check that the computational results are correct, using the Cushman-Sanders test \cite{MR855083,sanders2007averaging}.
We remark that the correctness argument involves two steps: first one shows that the candidate description of \(\ker M\) does not contain any linear relations,
then one computes the generating function to show that the dimensions at all degrees are correct, comparing the result to the known generating functions of general polynomials and polynomial vector fields.

The explicit construction of an 
$\Sl$-triple for general maps with nilpotent linear part is the
main result of this paper. 
In the vector field case, see \cite[Section 9.1]{sanders2007averaging}, the operator is also linear with respect to the linear
part of the vector field, actually a Lie algebra homomorphism. 
An embedding of the linear part therefore induces a 
$\Sl$-triple in the homological operator. 

The challenge for the case of maps is the nonlinearity with respect to the
linear part. 
Therefore, we indeed have to construct 
$\Sl$-triples for each homogeneous degree
separately. 
We are able to generalise this to all degrees such that
the construction only depends on the nilpotency index of the linear part.
The nilpotency index for a nilpotent matrix
$N\in \mathfrak{gl}(\mathbb{R})$ is defined as the least number 
$p$ such that 
$N^{p}=0$, where 
$\mathfrak{gl}_{n}(\mathbb{R})$ is the space of real valued
$n\times n$ matrices. 

The construction of 
$\Sl$-triples is first shown for nilpotent matrices in Jordan form.

Our approach is to first split the homological operator into two commuting operators by considering the map as an element in the tensor product of a space of polynomials and a vector space.
This allows us to embed them in two different 
$\Sl$-triples and from these
two we construct an  
$\Sl$-triple for the homological operator. 
The first operator is a multiplication operator which is a Lie algebra
homomorphism and the second is a substitution operator. 
Below we give a more detailed treatment of the matter. 
Problems arise with the second operator, which only is an antihomomorphism. 
Therefore our construction is specifically tailored to work with homomorphisms
and antihomomorphisms.
This is first done for matrices and then it is shown how the interplay with
homomorphisms allows for the general construction with relative ease.

The paper  is organized as follows. In Section \ref{sec:prelim} we describe the normal form procedure for maps.
In Section \ref{sec:weakinv} we find relations between the nilpotent and its transpose, followed by the construction of an \(\Sl\)-triple in Section \ref{sec:sl2triple}.
In Section \ref{sec:NFirreducible} this triple is used to compute the normal form in  the general irreducible case using transvectants,
with the \(2\)- and \(3\)-dimensional case as explicit example.
In Section \ref{sec:reducible} we treat several reducible cases with two blocks, namely the \((2,3)\)- and \((2,2)\)-case and the \((k_1,k_2)\)-case.
Here we show the effectiveness of the \(\Sl\) representation theory to prove the correctness of the computed normal form.

The results in the Sections \ref{sec:prelim}-\ref{sec:sl2triple} are based on the MSc-thesis \cite{roell2020}.
\section{Preliminaries}\label{sec:prelim}
Consider a smooth map 
$f:\mathbb{R}^{n}\to \mathbb{R}^{n}$, such that 
$f(0)=0$.

Define the vector space 
$P_{k}$ as the span of the set of homogeneous monomials of  degree 
$k$ and let \(\mathcal{P}^n_k=P_{k+1}\otimes \mathbb{R}^n\).

Then set  
$\mathcal{P}^n = \prod_{k\in \mathbb{N}} \mathcal{P}^n_{k}$, the direct product of the homogeneous
spaces (product, since we allow infinite summations). 
With this notation, any smooth map 
$f\in \mathcal{P}^n$ with 
$f(0)=0$ can be written as 
\begin{equation}
\label{eqn:generalmap}
f = Ax +\sum_{i=1}^\infty f_{i}(x) , 
\quad f_{i}\in \mathcal{P}^n_{i},
\end{equation}
where 
$f_{i}$ are the terms of degree 
$i+1$ in the multivariate Taylor expansion of 
$f$.
Without loss of generality we assume that the linear part  
$A:=Df(0)$ is in some sort of desired form, usually real block diagonal form or  
Jordan normal form.

A near identity transformation is a map  
$\varphi\in \mathcal{P}^n$ such that  
$\varphi(0)=0, D\varphi(0)=I$ and any near identity transformation can therefore be written as  
\[
\varphi(x) = x + \varphi_{1}(x)+ \varphi_{2}(x) + \ldots, \quad
\varphi_{i}\in \mathcal{P}^n_{i}.
\]

Let 
$A\in \mathfrak{gl}_n(\mathbb{R})$ and define the map 
$\subs{A}:\mathcal{P}^n\to \mathcal{P}^n$ as 
\ba
\label{eqn:SubsOp}
(\subs{A}\varphi)(x) = \varphi(Ax),
\ea
where 
$\mathcal{P}^n$ is the space of vector polynomials. 
The following definition plays a key role in the theory of normal forms and has
a direct analogue in the vector field case. 
\begin{defn}
\label{def:HomologicalOperator}
The linear operator,
\ba
\conn{A}:P_{k+1}\otimes \mathbb{R}^n\to P_{k+1}\otimes \mathbb{R}^n,\quad \conn{A}=1\otimes \mult{A}-\subs{A}\otimes 1,
\label{eqn:HomologicalOperator}
\ea
is called the homological operator. 
This notation is inspired by the fact that if we consider the tensorproduct \(\mathbb{R}[x_1,\ldots,x_n]\otimes\mathbb{R}^n\),
then this is similar to the notation \(\Delta(X)=X\otimes 1 +1 \otimes X\) that is commonly used in the context of actions on tensor products.
\end{defn}
Consider a near identity transformation
$\varphi\in \mathcal{P}^n$.
Then any map \eqref{eqn:generalmap} can be transformed to a new map 
$\bar{f}$ via 
$\bar{f}=\varphi\circ f \circ \varphi^{-1}$.
By expanding the right hand side of 
$\bar{f}$ and collecting terms of the same degree, we can write the
transformed map as 
\ba
	\label{eqn:TransformedSystem}
\bar{f}(x)
= Ax 
+ f_{1}(x) - (\conn{A}\varphi_{1})(x) 
+ f_{2}(x) - (\conn{A}\varphi_{2})(x) + r_{2}(x) 
+ \ldots, 
\ea
where 
$f_{i} \in \mathcal{P}^n_{i}$ denote the terms in the original map, 
$\conn{A}$ is defined as in \eqref{eqn:HomologicalOperator} and the terms 
$r_{i} \in \mathcal{P}^n_{i}$ are additional terms that depend on  
$\varphi_{k},\ 1 \le k < i$ and 
$f_{k},\ 1\le k < i$. 

Each homogeneous space 
$\mathcal{P}^n_{k}$ of vector monomials can be decomposed as 
$\mathcal{P}^n_{k}=\im\conn{A}\oplus \mathcal{C}_{k}$ where 
$\mathcal{C}_{k}$ is a complement to 
$\im \conn{A}$, determining the {\em style} \(\mathcal{C}\) of the normal form.
Any term in 
$f_{i}\in \mathcal{P}^n_{i}$ that lies in the image of 
$\conn{A}$ can be transformed away  by a choosing 
$\varphi_{i}\in \mathcal{P}^n_{i}$ such that 
$\conn{A}\varphi_{i}=f_{i}$.
Therefore the normal form of a map 
$f:\mathbb{R}^{n}\to \mathbb{R}^{n}$ with respect to the linear part 
$A$ can be defined as follows.

\begin{defn}
\label{def:NormalForm}
The map \eqref{eqn:generalmap} is said to be in normal form (with style \(\mathcal{C}\)) with respect to its linear part 
$A$ if 
$f_{k} \in \mathcal{C}_{k}, \forall k \in \mathbb{N}$ where 
$\mathcal{P}^n_{k} = \im \conn{A} \oplus\mathcal{C}_{k}$ and 
$\conn{A}$ is defined as in \eqref{eqn:HomologicalOperator}. 
\end{defn}
\begin{rem}
The terminology {\em style} has been introduced by James Murdock and tries to convey that it is a matter of choice, at times depending on taste and fashion.
\end{rem}
From \eqref{eqn:TransformedSystem}, the normal form of 
$f$ can be found by recursively solving the equations 
\bas
\conn{A}\varphi_{1} &=& f_{1}   - \bar{f}_{1},\\  
\conn{A}\varphi_{2} &=& f_{2} + r_{2} - \bar{f}_{2},\\  
\conn{A}\varphi_{3}& =& f_{3} + r_{3} - \bar{f}_{3},\\
&\vdots&
\eas
for some general  
$\bar{f}_{k}\in \mathcal{C}_{k}$ and 
$\varphi_{k}\in \mathcal{P}^n_{k}$.
The terms 
$\bar{f}_{k}$ that remain constitute the normal form of 
$f$.

Then \eqref{eqn:HomologicalOperator} can be written as
\ba
\label{eqn:SplitHomOp}
\conn{A} = \mult{A} - \subs{A}. 
\ea

The following lemma provides some properties of 
$\subs{A}$  
as defined by Equation  \eqref{eqn:SubsOp}.
\begin{lem}\label{lem:PropSubsMult}
The  following holds true
for all 
$A,B\in \mathfrak{gl}(\mathbb{R})$.
\renewcommand{\labelenumii}{\Roman{enumii}}
\begin{itemize}
\item [\rm 1]
The maps
$\mult{A}$ and 
$\subs{B}$ commute, or equivalently \( [\mult{A},\subs{B}] = 0\).
\item  [\rm 2]
For any
$k\in \mathbb{N}$,
\(
(\subs{A})^{k} = \subs{(A^k)},
\)
and we can write this as \(\subs{A}^k\).
\item  [\rm 3] If 
$A$ is nilpotent,  then 
$\subs{A}$ 
is nilpotent.
\item  [\rm 4]
The map $A\mapsto \subs{A}$
is a algebra antihomomorphism
\(
\subs{A}\subs{B}= \subs{(BA)}.
\)
\noindent
\end{itemize}
This does not imply it is a Lie algebra antihomomorphism.
\end{lem}

\begin{proof}
For  the  first item we have then
\bas
([\mult{A},\subs{B}]\varphi)(x)
&=&
\mult{A}\subs{B}\varphi(x)-\subs{B}\mult{A}\varphi(x)
= 
A\varphi(Bx) - A\varphi(Bx)=0.
\eas
For the second item 
\bas
(\subs{A})^{k}\varphi(x)
=
(\subs{A}^{k-1}\varphi)(Ax)
=
\cdots
=
\varphi(A^{k}x)
=
\subs{(A^{k})}\varphi(x).
\eas
For the third  one, 
since
$A^{p}=0$, then 
for 
an arbitrary $\varphi\in \mathcal{P}^n$ and  
by using  the  previous   item we obtain  
\bas
(\subs{A}^{p}\varphi)(x)=(\subs{(A^{p})}\varphi)(x)=\varphi(A^{p}x)=\varphi(0)=0.
\eas
As 
$\varphi$ was arbitrary, we find that 
$\subs{A}^{p}$ maps every element to zero, and hence is the
zero operator. 
For the last one  by direct computation  we have 
 then
\bas
\subs{A}\subs{B}\varphi(x)
=
\subs{B}\varphi(Ax)
=
\varphi(BAx)
=
\subs{(BA)}\varphi(x),
\eas
\end{proof}
\begin{cor}\label{cor:nilp}
If \(A^p=0\) then \(\conn{A^{2p}}=0\).
\end{cor}
We state some important properties of 
the homological operator in the following lemma. 

\begin{lem}[{\cite[Lemma 2.1]{gramchev2005normal}}]
Let 
$A,B\in \mathfrak{gl}_{n}(\mathbb{R})$ and let 
$\conn{A}$ be as in \eqref{eqn:HomologicalOperator}.  
Then 
\begin{itemize}
\item[(a)] 
If 
$A$ is nilpotent, then 
$\conn{A}$ is nilpotent (see Corollary \ref{cor:nilp}). 
\item[(b)] If 
$AB=BA$, then 
$\conn{A}
\conn{B}=
\conn{B}
\conn{A}$. 
\end{itemize}
\label{lem:PropertiesHomologicalOperator}
\end{lem}
In principle, one can rely on the Jacobson-Morozov theorem to extend a nilpotent matrix to an \(\Sl\)-triple.
The main difficulty lies in the embedding of 
$\subs{\mathbf{n}}$ into a triple, because the map 
$A\to \subs{A}$ for 
$A\in \mathfrak{gl}_{n}(\mathbb{R})$ is {\em not} a linear map (implicitly assuming it is, has simplified some `proofs' in the literature).
One could for fixed dimension and degree compute the matrix of \(\subs{\mathbf{n}}\) and then apply Jacobson-Morozov,
but we prefer to give a construction for which we can vary the dimension and get explicit formulas for the \(\Sl\)-triple.
This has the advantage of having a uniform definition of the normal form style.

We hasten to point out that this is the main difficulty that needs to be
overcome.
To illustrate that the map 
$A\mapsto \conn{A}$ is not a Lie algebra homomorphism, one computes for 
$A,B\in \mathfrak{gl}_{n}(\mathbb{R})$,
\begin{eqnarray*}
\conn{[A,B]}=
 \mult{[A,B]} - \subs{[A,B]}.
\end{eqnarray*}
On the other hand the bracket of 
$\conn{A}$ and 
$\conn{B}$ expands as
\begin{eqnarray*}
[\conn{A},\conn{B}]&=&(\mult{A}-\subs{A})(\mult{B}-\subs{B})-(\mult{B}-\subs{B})(\mult{A}-\subs{A})
\\&=& \mult{A}\mult{B}-\mult{A}\subs{B}-\subs{A}\mult{B}+\subs{A}\subs{B}-\mult{B}\mult{A}+\mult{B}\subs{A}+\subs{B}\mult{A}-\subs{B}\subs{A}
\\&=&\mult{A}\mult{B}+\subs{A}\subs{B}-\mult{B}\mult{A}-\subs{B}\subs{A}
\\&=&\mult{(AB)}-\mult{(BA)}+[\subs{A},\subs{B}]
\\&=&\mult{[A,B]}+[\subs{A},\subs{B}].
\end{eqnarray*}
The two quantities are not equal because 
$\varphi$ is not a linear map and hence we cannot conclude that
$(\conn{[A,B]}\varphi)(x)$ is in general equal to $ ([\conn{A},\conn{B}]\varphi)(x)$,
since \(\varphi((AB-BA)x)\) is not in general equal to \(\subs{(AB)}\varphi(x)-\subs{(BA)}\varphi(x)\). 
Therefore the map 
$A\mapsto \conn{A}$ is not a Lie algebra homomorphism.

The following definition is used to construct a triple for the nilpotent matrix 
$\mathbf{n}\in \mathfrak{gl}_{n}(\mathbb{R})$.

\begin{defn}
\label{def:m}
Let 
$\mathbf{n}\in \mathfrak{gl}_{n}(\mathbb{R})$ be a nilpotent matrix and let 
$\nN$ be the Jordan normal form of 
$\mathbf{n}$, the block diagonal matrix with nilpotent 
Jordan blocks on the block diagonal. 
Let 
$P\in GL_{n}(\mathbb{R})$ be an invertible matrix such that 
$\mathbf{n}=P^{-1}\nN P$.
Then we define 
$\mathbf{m}=P^{-1}\nM P$, where \(\nM\) is the transpose of \(\nN\).
Notice that \(\mathbf{n}^t=P^t P \mathbf{m} (P^t P )^{-1}\),
so unless \(P\) is orthogonal, \(\mathbf{m}\) is not the transpose of \(\mathbf{n}\).
We call \(\mathbf{m}\) the {\em conjugate transpose} of \(\mathbf{n}\).
\end{defn}

\begin{remark}
\label{rmk:Triples}
The nilpotent linear part of the map
$f\in \mathcal{P}^n$ to be normalised will be denoted by 
$\mathbf{n}$. 
The main goal is to embed the homological operator 
$\conn{\mathbf{n}}$, see Definition \ref{def:HomologicalOperator}, into a triple. 
We mention here that the construction of such an \(\Sl\)-triple is not trivial and the formulas may look cumbersome.
On the bright side, once the triple is in place, almost all our computations can be done with $\conn{\mathbf{n}}$.
The triple then defines the normal form style, but is not needed in the actual computations, since we then rely on the abstract \(\Sl\)-representation theory
and the Clebsch-Gordan formalism.

\end{remark}

First we will provide an explicit construction of 
$\mathfrak{sl}_{2}$-triples for nilpotent matrices.
Then we use the fact that 
$\subs{\mathbf{n}}$ is an antihomomorphism to construct an 
$\mathfrak{sl}_{2}$-triple for 
$\subs{\mathbf{n}}$. 
\section{Relations}\label{sec:weakinv}


By convention, the Jordan normal form of a matrix 
$A\in \mathfrak{gl}_{n}(\mathbb{R})$ is the upper triangular matrix
with the eigenvalues on the main diagonal, either ones or zeros on the super
diagonal and zeros elsewhere.
\begin{lem}
\label{lem:WeakInvPower}
Let 
$\nN\in \mathfrak{gl}_{n}(\mathbb{R})$ be a nilpotent matrix in Jordan form and 
$\nM$ be its transpose. 
Then  \(\nN\nM\nN=\nN\) and \(\nM\nN\nM=\nM\).
\end{lem}

\begin{proof}
Let
$\nN\in \mathfrak{gl}_{n}(\mathbb{R})$ consist of a single nilpotent Jordan block.
By block diagonality, it is sufficient to prove the result for one
nilpotent Jordan block.

A straightforward computation shows that 
\begin{eqnarray}
\label{eqn:mnnm}
\nN\nM\nN
&=&
\begin{pmatrix}
0_{n-1,1}&I_{n-1} \\
0&  0_{1,n-1}
\end{pmatrix}
\begin{pmatrix}
0_{1,n-1} & 0\\
I_{n-1}& 0_{n-1,1} 
\end{pmatrix}
\begin{pmatrix}
0_{n-1,1}&I_{n-1} \\
0&  0_{1,n-1}
\end{pmatrix}
\\&=&
\begin{pmatrix}
I_{n-1} & 0_{n-1,1}\\
0_{1,n-1}&  0
\end{pmatrix}
\begin{pmatrix}
0_{n-1,1}&I_{n-1} \\
0&  0_{1,n-1}
\end{pmatrix}
=\begin{pmatrix}
0_{n-1,1}&I_{n-1} \\
0&  0_{1,n-1}
\end{pmatrix}
=\nN,
\end{eqnarray}
where 
$I_{k}$ is the 
$k\times k$ identity and 
$0_{l,k}$ is the 
$l\times k$-matrix filled with zeros.  This proves the first identity.
The second follows by transposition.
\end{proof}
\begin{thm}
Let \(\mathbf{n}\) be nilpotent and \(\mathbf{m}\) its conjugate transpose. Then \(\mathbf{n}\mathbf{m}\mathbf{n}=\mathbf{n}\) and \(\mathbf{m}\mathbf{n}\mathbf{m}=\mathbf{m}\).
\end{thm}
\begin{proof}
Let 
$P\in GL_{n}(\mathbb{R})$ be the invertible matrix such that 
$\mathbf{n}=P^{-1}\nN P$.
By Definition \ref{def:m},
$\mathbf{m}=P^{-1}\nM P$.
Then \(\mathbf{n}\mathbf{m}\mathbf{n}=P^{-1}\nN P P^{-1}\nM P P^{-1}\nN P=P^{-1}\nN\nM\nN P=P^{-1}\nN P=\mathbf{n}\)
and \(\mathbf{m}\mathbf{n}\mathbf{m}=P^{-1}\nM P P^{-1}\nN P P^{-1}\nM P=P^{-1}\nM \nN \nM P=P^{-1}\nM P=\mathbf{m}\).
\end{proof}
In the proof of Theorem \ref{thm:SL2Mat}, we will make extensive use of the
following relations. 
\begin{thm}
\label{thm:RelasMN}
Let 
$\nN\in \mathfrak{gl}_{n}(\mathbb{R})$ consist of a single nilpotent Jordan
block and let 
$\nM$ be its transpose.  
Then for any 
$0\leq l \leq k, 1\leq k\in \mathbb{N}$ the following relations hold. 
\begin{itemize}
\item [a)] \label{thm:RelasJN1}
$\nN^l\nM^{k}\nN^{k}=\nM^{(k-l)}\nN^{k}$, 
\item [b)] \label{thm:RelasJN2}
$\nM^{k}\nN^{k}\nM^{l}=\nM^{k}\nN^{k-l}$.
\end{itemize}
\end{thm}

\begin{proof} 
For 
$k\ge n$, the left hand side and right hand side are zero in
(a) and it is only necessary to prove the result for 
$k< n$.
Assume that 
$\nN$ is an 
$n\times n$ matrix. 
A computation shows that for 
$k < n $,
\begin{equation}
\nN^k=\begin{pmatrix}
0_{n-k,k} &I_{n-k} \\
O_{k} & O_{n-k}
\end{pmatrix},
\end{equation}
where \(O_k\) is the zeromatrix, 
Denote by 
$J_{k}$ a single 
$k\times k$ nilpotent Jordan block and
$N_{n,k}$ the 
$n\times k$ zero matrix with a one in the
lower left corner. 
In the special case that 
$n=k=1$, we define 
$N_{1,1}=(1)$ and 
$J_{1}=(0)$. 
\begin{eqnarray}
\label{eqn:powers}
\nM^{k}\nN^k
&=&
\begin{pmatrix}
0_{k,n-k} & O_k\\
I_{n-k}& 0_{n-k,k}
\end{pmatrix}
\begin{pmatrix}
0_{n-k,k}&I_{n-k} \\
O_k&  0_{k,n-k}
\end{pmatrix}
=
\begin{pmatrix}
O_{k} & 0_{k,n-k}\\
0_{n-k,k}&  I_{n-k}
\end{pmatrix}.
\end{eqnarray}
For  \(l=0\) the  statement is trivial. Now, we prove  the statement for \(l=1\)
by computing  the left hand side and right hand side
of the equality to show the result. 
Then 
$\nN$ can be partitioned as 
\begin{equation}
\label{eqn:part}
\nN= 
\begin{pmatrix}
J_{k}&  N_{k,n-k}\\
0_{n-k,k}&  J_{n-k}
\end{pmatrix}
.
\end{equation}
Using this partition and \eqref{eqn:powers}, we have
\begin{eqnarray*}
\nN \nM^{k}\nN^k
&= &
\begin{pmatrix}
J_{k}&  N_{k,n-k}\\
0_{n-k,k}&  J_{n-k}
\end{pmatrix}
\begin{pmatrix}
O_{k}& 0_{k,n-k}     \\
0_{n-k,k}& I_{n-k}
\end{pmatrix}
= 
\begin{pmatrix}
O_{k}& N_{k,n-k}\\
0_{n-k,k}& J_{n-k}
\end{pmatrix}
\\&=& 
\begin{pmatrix}
O_{k-1}& 0_{k-1,n-k+1}\\
0_{n-k+1,k-1} &J_{n-k+1}
\end{pmatrix}
=\nM^{k-1} \nN^k.
\end{eqnarray*}
Now the result follows by induction on \(l\).
The   second item of the theorem follows  by transposing the  first item. 
\end{proof}
The above Theorem readily generalizes to the case of a general nilpotent
matrix. 

\begin{cor}\label{cor:rels}
Let 
$\mathbf{n}$ be any nilpotent matrix  and let
$\mathbf{m}$ be its conjugate transpose (as defined in Definition \ref{def:m}). Then
\begin{itemize}
\item [a)] 
$\mathbf{n}^l\mathbf{m}^{k}\mathbf{n}^{k}=\mathbf{m}^{(k-l)}\mathbf{n}^{k}$, 
\item [b)] 
$\mathbf{m}^{k}\mathbf{n}^{k}\mathbf{m}^{l}=\mathbf{m}^{k}\mathbf{n}^{k-l}$.
\end{itemize}
Since we can switch the role of \(\mathbf{n}\) and \(\mathbf{m}\) we also have
\begin{itemize}
\item [c)] 
$\mathbf{m}^l\mathbf{n}^{k}\mathbf{m}^{k}=\mathbf{n}^{(k-l)}\mathbf{m}^{k}$, 
\item [d)] 
$\mathbf{n}^{k}\mathbf{m}^{k}\mathbf{n}^{l}=\mathbf{n}^{k}\mathbf{m}^{k-l}$.
\end{itemize}
Furthermore, taking \(l=k\) we see that \(\mathbf{n}^k\mathbf{m}^k\mathbf{n}^k=\mathbf{n}^k\) and \(\mathbf{m}^k\mathbf{n}^k\mathbf{m}^k=\mathbf{m}^k\).
\end{cor}
\begin{lem}
For \(l\geq k\) we have \(\mathbf{n}^l\mathbf{m}^{k}\mathbf{n}^{k}=\mathbf{n}^{l}\) and \(\mathbf{m}^{k}\mathbf{n}^{k}\mathbf{m}^{l}=\mathbf{m}^{l}\).
\end{lem}
\begin{proof}
Assume \(l\geq k\). Then
\(\mathbf{n}^l\mathbf{m}^{k}\mathbf{n}^{k}=\mathbf{n}^{(l-k)+k}\mathbf{m}^{k}\mathbf{n}^{k}=\mathbf{n}^{(l-k)}\mathbf{n}^{k}=\mathbf{n}^{l}\)
and \(\mathbf{m}^{k}\mathbf{n}^{k}\mathbf{m}^{l}=\mathbf{m}^{k}\mathbf{n}^{k}\mathbf{m}^{k+(l-k)}=\mathbf{m}^{l}\).
\end{proof}
\begin{cor}
\label{cor:allrels}
Let 
$\mathbf{n}$ be any nilpotent matrix  and let
$\mathbf{m}$ be its conjugate transpose and assume \(k,l \in\mathbb{N}, k\geq 1\).
Let \(k\oplus l=\max(k,l)\). Then
\begin{itemize}
\item \(\mathbf{n}^l\mathbf{m}^{k}\mathbf{n}^{k}=\mathbf{m}^{\max(0,k-l)}\mathbf{n}^{\max(k,l)} =\mathbf{m}^{(k\oplus l)-l} \mathbf{n}^{k\oplus l}\),
\item \(\mathbf{m}^{k}\mathbf{n}^{k}\mathbf{m}^{l}=\mathbf{m}^{\max(k,l)}\mathbf{n}^{\max(0,k-l)}=\mathbf{m}^{(k\oplus l)} \mathbf{n}^{(k\oplus l)-l}\).
\item \(\mathbf{m}^l\mathbf{n}^{k}\mathbf{m}^{k}=\mathbf{n}^{\max(0,k-l)}\mathbf{m}^{\max(k,l)} =\mathbf{n}^{(k\oplus l)-l} \mathbf{m}^{k\oplus l}\),
\item \(\mathbf{n}^{k}\mathbf{m}^{k}\mathbf{n}^{l}=\mathbf{n}^{\max(k,l)}\mathbf{m}^{\max(0,k-l)}=\mathbf{n}^{(k\oplus l)} \mathbf{m}^{(k\oplus l)-l}\).
\end{itemize}
\end{cor}
\begin{cor}
\label{cor:allrelsstar}
Let
$\mathbf{n}$ be any nilpotent matrix  and let
$\mathbf{m}$ be its conjugate transpose and assume \(k,l \in\mathbb{N}, k\geq 1\).
Let \(k\oplus l=\max(k,l)\). Then
\begin{itemize}
\item \(\subs{\mathbf{n}}^l\subs{\mathbf{m}}^{k}\subs{\mathbf{n}}^{k}=\subs{(\mathbf{n}^{k}\mathbf{m}^{k}\mathbf{n}^l)}=\subs{(\mathbf{n}^{(k\oplus l)} \mathbf{m}^{(k\oplus l)-l})}=\subs{\mathbf{m}^{(k\oplus l)-l}}\subs{\mathbf{n}^{(k\oplus l)} }\)
\item \(\subs{\mathbf{m}}^{k}\subs{\mathbf{n}}^{k}\subs{\mathbf{m}}^{l}=\subs{(\mathbf{m}^l\mathbf{n}^{k}\mathbf{m}^{k})}=\subs{(\mathbf{n}^{(k\oplus l)-l} \mathbf{m}^{k\oplus l})}=\subs{ \mathbf{m}}^{k\oplus l}\subs{\mathbf{n}^{(k\oplus l)-l}}\).
\item \(\subs{\mathbf{m}^l}\subs{\mathbf{n}^{k}}\subs{\mathbf{m}^{k}}=\subs{(\mathbf{m}^{k}\mathbf{n}^{k}\mathbf{m}^{l})}=\subs{(\mathbf{m}^{(k\oplus l)} \mathbf{n}^{(k\oplus l)-l})}=\subs{\mathbf{n}^{(k\oplus l)-l}}\subs{\mathbf{m}^{(k\oplus l)}}\),
\item \(\subs{\mathbf{n}^{k}}\subs{\mathbf{m}^{k}}\subs{\mathbf{n}^{l}}=\subs{(\mathbf{n}^l\mathbf{m}^{k}\mathbf{n}^{k})}=\subs{(\mathbf{m}^{(k\oplus l)-l} \mathbf{n}^{k\oplus l})}=
\subs{ \mathbf{n}^{k\oplus l}}\subs{\mathbf{m}^{(k\oplus l)-l}}\).
\end{itemize}
\end{cor}
\begin{lem}\label{lem:kermmm}
Let
$\mathbf{n}$ be any nilpotent matrix  and let
$\mathbf{m}$ be its conjugate transpose and assume \(i,l \in\mathbb{N}, i\geq 1\).
Let \(\pi_i^l=\mathbf{n}^{l}\mathbf{m}^{i}\mathbf{n}^{i-l}\).
Then \(\pi_i^l\) is a projection operator.
\end{lem}
\begin{proof}
\bas
\pi_i^l\cdot \pi_i^l&=& \mathbf{n}^{l}\mathbf{m}^{i}\mathbf{n}^{i-l}\mathbf{n}^{l}\mathbf{m}^{i}\mathbf{n}^{i-l}
\\&=& \mathbf{n}^{l}\mathbf{m}^{i}\mathbf{n}^{i}\mathbf{m}^{i}\mathbf{n}^{i-l}
\\&=& \mathbf{n}^{l}\mathbf{m}^{i}\mathbf{n}^{i-l}
\\&=&\pi_i^l.
\eas
\end{proof}
We need the following Lemma in the proof of Lemma \ref{lem:proj}.
\begin{lem}\label{lem:kermm}
Let \(\nN\) be in Jordan normal form and let \(P_i^l=\nN^{l}\nM^{i}\nN^{i-l}, p>i>l\geq 0\).
Then \(P_i^l\) is a diagonal projection operator and it projects on the \(p-i\)-dimensional space spanned by \(\langle e_{i-l+1},\cdots,e_{p-l}\rangle\).
In particular, \(\varpi_{l}=P_{p-1}^{p-l}=\nN^{p-l}\nM^{p-1}\nN^{l-1}\) projects on \(e_{l}\) and the \(\varpi_{l}, l=1,\cdots,p\) are complete: \[
\sum_{l=1}^{p} \varpi_{l}=id_p.\]
Furthermore, \(P_i^l=\sum_{j=i-l+1}^{p-l} \varpi_j\)
and \(\mathcal{E}_p:=1+\sum_{i=2}^{p-1} \sum_{l=0}^{i-1} P_i^l=\varpi_{1}+\sum_{j=2}^{p} (p-j+1) (j-1) \varpi_j\).
\end{lem}
\begin{proof}
The projection part follows from Lemma \ref{lem:kermmm}.
Diagonality follows from the Jordan normal form of \(\nN\) and \(\nM\).
Then, assuming \(\nN\) is irreducible,
\bas
\mathcal{E}_p&=&1+\sum_{i=2}^{p-1} \sum_{l=0}^{i-1} P_i^l 
=
\sum_{j=1}^{p} \varpi_{j}+\sum_{i=2}^{p-1} \sum_{l=0}^{i-1} \sum_{j=i-l+1}^{p-l} \varpi_j
=
\varpi_{1}+\sum_{i=1}^{p-1} \sum_{l=0}^{i-1} \sum_{j=i-l+1}^{p-l} \varpi_j=
\\&=&
\varpi_{1}+\sum_{l=0}^{p-2} \sum_{i=l+1}^{p-1} \sum_{j=i-l+1}^{p-l} \varpi_j
=
\varpi_{1}+\sum_{l=0}^{p-2} \sum_{j=2}^{p-l} \sum_{i=l+1}^{j+l-1} \varpi_j
=
\varpi_{1}+\sum_{l=0}^{p-2} \sum_{j=2}^{p-l} (j-1) \varpi_j=
\\&=&
\varpi_{1}+\sum_{j=2}^{p} \sum_{l=0}^{p-j} (j-1) \varpi_j
=
\varpi_{1}+\sum_{j=2}^{p} (p-j+1) (j-1) \varpi_j,
\eas
with trace 
\(1+\frac{1}{6}p(p+1)(p-1)\). 
\end{proof}
\begin{lem}\label{lem:traceP}
\(\mathrm{Tr} P_i^l=p-i\).
\end{lem}
\begin{proof}
Since \(\nM^{i}\nN^{i}\) projects on the last \(p-i\) coordinates, one finds
\bas
\mathrm{Tr}(P_i^l)&=&\mathrm{Tr}(\nN^{l}\nM^{i}\nN^{i-l})=\mathrm{Tr}(\nM^{i}\nN^{i})=p-i.
\eas
\end{proof}

\section{An $\mathfrak{sl}_{2}$-triple for a nilpotent matrix}\label{sec:sl2triple}
In the following Theorem \ref{thm:SL2Mat} an explicit construction of an 
$\mathfrak{sl}_{2}$-triple for a nilpotent matrix 
$\mathbf{n}\in \mathfrak{gl}_{n}(\mathbb{R})$ is given, using its conjugate transpose \(\mathbf{m}\) as defined in Definition \ref{def:m}.
It follows from Lemmas \ref{lem:kermm}
and \ref{lem:proj}
that this construction coincides with the triples
constructed in \cite[Section 2.5]{murdock2006normal} for nilpotent matrices in  Jordan form.

Before we can prove the theorem, we need the technical Lemma \ref{lem:cases}.
In the proof of Theorem \ref{thm:SL2Mat}, the computation of the relation 
$[\overline{\mathbf{h}},\overline{\mathbf{m}}]=2\overline{\mathbf{m}}$ is done
by showing that this relation holds for the generators of 
$\mathbf{m}$. 
The generators are given by the expression 
$\mathbf{n}^{l}\mathbf{m}^{i}\mathbf{n}^{i-l-1}$ and the Lie bracket to be
computed is given by 
\begin{equation}
\left[
\left[ 
\mathbf{m}^{k},
\mathbf{n}^{k}
\right],
\mathbf{n}^{l}\mathbf{m}^{i}\mathbf{n}^{i-l-1}
\right].
\label{eqn:bracket}
\end{equation}
Here 
$\mathbf{n}\in \mathfrak{gl}_{n}(\mathbb{R})$ is a nilpotent matrix with nilpotency index 
$p\in \mathbb{N}$ and 
$\mathbf{m}$ is as in Definition \ref{def:m}. 
The powers of 
$\mathbf{n}$ and 
$\mathbf{m}$ satisfy the relations
$1\le i\le p$,
$0\le l \le i$ and 
$1\le k\le p$.
The following Lemma computes the Lie bracket \eqref{eqn:bracket}.

\begin{lem}
\label{lem:cases}
Let 
$\mathbf{n}\in  \mathfrak{gl}_{n}(\mathbb{R})$ be a nilpotent matrix with nilpotency
index 
$p$ and 
$\mathbf{m}$ be as in Definition \ref{def:m}. 
Fix 
$i,l\in \mathbb{N}$ such that 
$0\le i,l \le p$.
Then for any natural number  
$1 \le k \le p$ the Lie bracket \eqref{eqn:bracket} equals either one
of the following.  
\begin{itemize}
\item[(a)] If 
$k\le \min\left\{ l,i-l-1 \right\}$ then 
\[
\left[ 
[\mathbf{m}^{k},\mathbf{n}^{k}], 
\mathbf{n}^{l}\mathbf{m}^{i}\mathbf{n}^{i-l-1}
\right]
= 0.
\]
\item[(b)] If 
$l<k\le i-l-1$, then 
\[
\left[ 
[\mathbf{m}^{k},\mathbf{n}^{k}], 
\mathbf{n}^{l}\mathbf{m}^{i}\mathbf{n}^{i-l-1}
\right]
= 
\mathbf{n}^{k-1}\mathbf{m}^{k+i-l-1}\mathbf{n}^{i-l-1}
-
\mathbf{n}^{k}\mathbf{m}^{k+i-l}\mathbf{n}^{i-l-1}.  
\]
\item[(c)]
If 
$i-l-1 < k \le l$, then 
\[
\left[ 
[\mathbf{m}^{k},\mathbf{n}^{k}], 
\mathbf{n}^{l}\mathbf{m}^{i}\mathbf{n}^{i-l-1}
\right]
= 
\mathbf{n}^{l}\mathbf{m}^{k+l}\mathbf{n}^{k-1} - 
\mathbf{n}^{l}\mathbf{m}^{k+l+1}\mathbf{n}^{k}.
\]
\item[(d)] 
If 
$ k > \max \left\{ l,i-l-1 \right\}$, then 
\[
\left[ 
[\mathbf{m}^{k},\mathbf{n}^{k}], 
\mathbf{n}^{l}\mathbf{m}^{i}\mathbf{n}^{i-l-1}
\right]
= 
\mathbf{n}^{k-1}\mathbf{m}^{k+i-l-1}\mathbf{n}^{i-l-1}
-
\mathbf{n}^{k}\mathbf{m}^{k+i-l}\mathbf{n}^{i-l-1}\\
+
\mathbf{n}^{l}\mathbf{m}^{k+l}\mathbf{n}^{k-1} - 
\mathbf{n}^{l}\mathbf{m}^{k+l+1}\mathbf{n}^{k}.
\]
\end{itemize}
\end{lem}
\begin{proof}
The proof proceeds by expanding \eqref{eqn:bracket} for each of the four
cases and then we use Corollary \ref{cor:rels} to rewrite each of the terms
to obtain the identities.
We first expand \eqref{eqn:bracket} as 
\begin{equation}
\mathbf{m}^{k}\mathbf{n}^{k+l}
\mathbf{m}^{i}\mathbf{n}^{i-l-1}
-
\mathbf{n}^{k}\mathbf{m}^{k} 
\mathbf{n}^{l}\mathbf{m}^{i}\mathbf{n}^{i-l-1}
-
\mathbf{n}^{l}\mathbf{m}^{i}\mathbf{n}^{i-l-1}
\mathbf{m}^{k}\mathbf{n}^{k}
+
\mathbf{n}^{l}\mathbf{m}^{i}\mathbf{n}^{i-l-1+k}
\mathbf{m}^{k}.
\label{eqn:bracket2}
\end{equation}
For the first term in \eqref{eqn:bracket2} we find 
\begin{eqnarray*}
\underline{\mathbf{m}^{k}\mathbf{n}^{k+l}} \mathbf{m}^{i}\mathbf{n}^{i-l-1}
&= &
\mathbf{n}^{l}\underline{\mathbf{m}^{k+l}\mathbf{n}^{k+l}\mathbf{m}^{i}}\mathbf{n}^{i-l-1}\\
&= &
\mathbf{n}^{l}\mathbf{m}^{(k+l)\oplus i}\mathbf{n}^{((k+l)\oplus i)-i}\mathbf{n}^{i-l-1}\\
&= &
\mathbf{n}^{l}\mathbf{m}^{(k+l)\oplus i}\mathbf{n}^{((k+l)\oplus i)-l-1}\\
&= &
\mathbf{n}^{l}\mathbf{m}^{(k+l)\oplus i}\mathbf{n}^{(k-1)\oplus (i-l-1)}
\end{eqnarray*}
where we underline terms that we combine to apply Corollary \ref{cor:allrels}.
If we assume that $k\le i-l-1 $, this reduces to \(\mathbf{n}^{l}\mathbf{m}^{i}\mathbf{n}^{i-l-1}\).
If we assume that $k> i-l-1 $, this reduces to \(\mathbf{n}^{l}\mathbf{m}^{k+l}\mathbf{n}^{k-1}\).

For the second term in \eqref{eqn:bracket2}  we find 
\begin{eqnarray*}
\underline{\mathbf{n}^{k}\mathbf{m}^{k}\mathbf{n}^{l}}\mathbf{m}^{i}\mathbf{n}^{i-l-1}
&= &
\mathbf{n}^{(k\oplus l)} \mathbf{m}^{(k\oplus l)-l+i}\mathbf{n}^{i-l-1}\\
&= &
\mathbf{n}^{k\oplus l} \mathbf{m}^{(k-l+i)\oplus i}\mathbf{n}^{i-l-1}\\
\end{eqnarray*}
If we assume that $k\le l $, this reduces to \(\mathbf{n}^{l}\mathbf{m}^{i}\mathbf{n}^{i-l-1}\).
\newline
If we assume that $k> l $, this reduces to \(\mathbf{n}^{k} \mathbf{m}^{k-l+i}\mathbf{n}^{i-l-1}\).

For the third term in \eqref{eqn:bracket2} we find 
\begin{eqnarray*}
\mathbf{n}^{l}\mathbf{m}^{i}\underline{\mathbf{n}^{i-l-1}\mathbf{m}^{k}\mathbf{n}^{k}}
&=& 
\mathbf{n}^{l}\mathbf{m}^{(k\oplus (i-l-1))+l+1} \mathbf{n}^{k\oplus (i-l-1)}
\\&=& 
\mathbf{n}^{l}\mathbf{m}^{(k+l+1)\oplus i} \mathbf{n}^{k\oplus (i-l-1)}
\end{eqnarray*}
If we assume that $k\le i-l-1 $, this reduces to \(\mathbf{n}^{l}\mathbf{m}^{i}\mathbf{n}^{i-l-1}\).
\newline
If we assume that $k> i-l-1 $, this reduces to \(\mathbf{n}^{l}\mathbf{m}^{k+l+1} \mathbf{n}^{k}\).

For the fourth term in \eqref{eqn:bracket2} we calculate
\begin{eqnarray*}
\mathbf{n}^{l}\mathbf{m}^{i}\mathbf{n}^{i-l-1+k}\mathbf{m}^{k}
&=&
\mathbf{n}^{l}\mathbf{m}^{i}\underline{\mathbf{n}^{i-l-1+k}\mathbf{m}^{k}}
\\&=&
\mathbf{n}^{l}\underline{\mathbf{m}^{i}\mathbf{n}^{k+i-l-1}\mathbf{m}^{k+i-l-1}}\mathbf{n}^{i-l-1}
\\&=&
\mathbf{n}^{l+\max(k-l-1,0)}\mathbf{m}^{i+\max(k-l-1,0)}\mathbf{n}^{i-l-1}
\end{eqnarray*}
If we assume that $k\le l $, this reduces to \(\mathbf{n}^{l}\mathbf{m}^{i}\mathbf{n}^{i-l-1}\).
But if \(l<k\), it reduces to \(\mathbf{n}^{k-1}\mathbf{m}^{i+k-l-1}\mathbf{n}^{i-l-1}\).

By adding the four terms with the appropriate signs and keeping track of the inequalities, we see that we have proved the Lemma.
\end{proof}

\begin{thm}
\label{thm:SL2Mat}
Let 
$\mathbf{n}\in \mathfrak{gl}_{n}(\mathbb{R})$ be a nilpotent matrix  with nilpotency index 
$p\ge 2$ and let 
$\mathbf{m}$ its conjugate transpose, as defined in Definition \ref{def:m}.
Then the triple 
$\langle \bar{\mathbf{n}},\bar{\mathbf{h}}, \bar{\mathbf{m}} \rangle$,
where  
\ba
\bar{\mathbf{n}}
:= 
\mathbf{n},
\quad
\bar{\mathbf{m}} 
:= 
\sum_{i=1}^{p-1} 
\sum_{l=0}^{i-1}
\mathbf{n}^{l}\mathbf{m}^{i}\mathbf{n}^{i-l-1},
\quad 
\bar{\mathbf{h}} 
:=
\sum_{i=1}^{p-1}
\left[
\mathbf{m}^{i},
\mathbf{n}^{i}
\right],
\label{eqn:SL2Mat}
\ea
is an 
$\mathfrak{sl}_{2}$-triple.
\end{thm}
\begin{rem}
Notice that one might change the \(p-1\) upper bound to \(\infty\); this does not change the definition, it does make it look more universal,
but obviously it adds a small worry on first reading.
\end{rem}
\begin{proof}
We verify that 
$\overline{\mathbf{n}}$, 
$\overline{\mathbf{m}}$ and 
$\overline{\mathbf{h}}$ satisfy the following relations
\[
[\overline{\mathbf{m}},\overline{\mathbf{n}} ] =
\overline{\mathbf{h}},
\quad
[\overline{\mathbf{h}},\overline{\mathbf{n}}] =
-2\overline{\mathbf{n}},
\quad 
[\overline{\mathbf{h}},\overline{\mathbf{m}}]=2\overline{\mathbf{m}}.
\label{eqn:TripleRelaMat}
\]
By a straightforward calculation one has 
\begin{eqnarray*}
[\overline{\mathbf{m}},\overline{\mathbf{n}}] 
&= &
\left[
\sum_{i=1}^{p-1}
\sum_{l=0}^{i-1}
\mathbf{n}^{l}\mathbf{m}^{i}\mathbf{n}^{i-1-l},
\mathbf{n}
\right] \\
&=& 
\sum_{i=1}^{p-1}
\sum_{l=0}^{i-1}
\mathbf{n}^{l}\mathbf{m}^{i}\mathbf{n}^{i-1-l}
\mathbf{n}
-
\sum_{i=1}^{p-1}
\sum_{l=0}^{i-1}
\mathbf{n}
\mathbf{n}^{l}\mathbf{m}^{i}\mathbf{n}^{i-1-l}\\
&= &
\sum_{i=1}^{p-1}
\left(
\sum_{l=0}^{i-1}
\mathbf{n}^{l}\mathbf{m}^{i}\mathbf{n}^{i-l}
-
\sum_{l=0}^{i-1}
\mathbf{n}^{l+1}\mathbf{m}^{i}\mathbf{n}^{i-1-l}
\right)\\
&=& 
[\mathbf{m},\mathbf{n}]
+
\sum_{i=2}^{p-1}
\left(
\mathbf{m}^{i}\mathbf{n}^{i}
+
\sum_{l=1}^{i-1}
\mathbf{n}^{l}\mathbf{m}^{i}\mathbf{n}^{i-l}
-
\sum_{l=0}^{i-2}
\mathbf{n}^{l+1}\mathbf{m}^{i}\mathbf{n}^{i-1-l} 
-
\mathbf{n}^{i}\mathbf{m}^{i} 
\right)\\
&=& 
[\mathbf{m},\mathbf{n}]
+
\sum_{i=2}^{p-1}
\left(
\mathbf{m}^{i}\mathbf{n}^{i}
+
\sum_{l=0}^{i-2}
\mathbf{n}^{l+1}\mathbf{m}^{i}\mathbf{n}^{i-1-l}
-
\sum_{l=0}^{i-2}
\mathbf{n}^{l+1}\mathbf{m}^{i}\mathbf{n}^{i-1-l} 
-
\mathbf{n}^{i}\mathbf{m}^{i} 
\right)\\
&= &
[\mathbf{m},\mathbf{n}]
+
\sum_{i=2}^{p-1}
\left[
\mathbf{m}^{i},\mathbf{n}^{i} 
\right]\\
&= &
\overline{\mathbf{h}}.
\end{eqnarray*}
For the second commutator relation in \eqref{eqn:TripleRelaMat}, we find
\begin{eqnarray*}
[\overline{\mathbf{h}},\overline{\mathbf{n}}]
&= &
\left[
\sum_{k=1}^{p-1} 
\left[ 
\mathbf{m}^{k},\mathbf{n}^{k} 
\right],
\mathbf{n}
\right]
= 
\sum_{k=1}^{p-1} 
\left(
\mathbf{m}^{k}\mathbf{n}^{k} 
-
\mathbf{n}^{k}\mathbf{m}^{k}
\right)
\mathbf{n}
-
\sum_{k=1}^{p-1} 
\mathbf{n}
\left(
\mathbf{m}^{k}\mathbf{n}^{k} 
-
\mathbf{n}^{k}\mathbf{m}^{k}
\right)\\
&= &
\sum_{k=1}^{p-1} 
\left(
\mathbf{m}^{k}\mathbf{n}^{k+1}
-
\mathbf{n}^{k}\mathbf{m}^{k}\mathbf{n}
-
\mathbf{n}\mathbf{m}^{k}\mathbf{n}^{k}
+
\mathbf{n}^{k+1}\mathbf{m}^{k}
\right)\\
&= &
\sum_{k=1}^{p-1} 
\left(
\mathbf{m}^{k}\mathbf{n}^{k+1}
-
\mathbf{n}^{k}\mathbf{m}^{k-1}
-
\mathbf{m}^{k-1}\mathbf{n}^{k}
+
\mathbf{n}^{k+1}\mathbf{m}^{k}
\right).
\end{eqnarray*}
The equality between the second-last and last step follows by applying
Corollary \ref{cor:rels} for \(l=1\). 
Notice that in the last expression the first and the third form a telescoping series as well as the
second and fourth term. 
The terms are rearranged and summation over 
$k$ yields
\begin{eqnarray*}
\left[
\,\overline{\mathbf{h}},
\overline{\mathbf{n}}
\right] 
&= &
\sum_{k=1}^{p-1} 
\left(
\mathbf{m}^{k}\mathbf{n}^{k+1}
-
\mathbf{n}^{k}\mathbf{m}^{k-1}
-
\mathbf{m}^{k-1}\mathbf{n}^{k}
+
\mathbf{n}^{k+1}\mathbf{m}^{k}
\right)\\
&= &
\sum_{k=1}^{p-1} 
\left(
\mathbf{m}^{k}\mathbf{n}^{k+1}
-
\mathbf{m}^{k-1}\mathbf{n}^{k}
-
\mathbf{n}^{k}\mathbf{m}^{k-1}
+
\mathbf{n}^{k+1}\mathbf{m}^{k}
\right)\\
&= &
\mathbf{m}^{p-1}\mathbf{n}^{p}
-
\mathbf{n}
+
\mathbf{n}^{p}\mathbf{m}^{p-1}
-
\mathbf{n}
= 
-2\mathbf{n}
= -2 \overline{\mathbf{n}},
\end{eqnarray*}
where the nilpotency of 
$\mathbf{n}$ implies
$\mathbf{n}^{p}=0$.\\

To prove the last equation in \eqref{eqn:TripleRelaMat},  it is sufficient by linearity, to prove that
\ba
\label{eqn:BracketRelation}
\left[
\,\overline{\mathbf{h}},
\mathbf{n}^{l}\mathbf{m}^{i}\mathbf{n}^{i-l-1}
\right]
= 
2\mathbf{n}^{l}\mathbf{m}^{i}\mathbf{n}^{i-l-1}
\ea
holds for all
$0\le l \le i-1$ and 
$0\le i <p$. 
To prove that 
\eqref{eqn:BracketRelation} holds, we consider three distinct cases. 
\begin{description}
\item[Case I] 
First we assume that 
$2l < i-1$ and expand the bracket as,
\begin{eqnarray}
\label{eqn:CommutatorHM}
\left[
\,\overline{\mathbf{h}},
\mathbf{n}^{l}\mathbf{m}^{i}\mathbf{n}^{i-l-1}
\right]
&=& 
\sum_{k=1}^{p-1}
\left[
\left[ 
\mathbf{m}^{k},\mathbf{n}^{k}
\right],
\mathbf{n}^{l}\mathbf{m}^{i}\mathbf{n}^{i-l-1}
\right]\\
&=& 
\sum_{k=1}^{l} 
\left[ 
\left[ 
\mathbf{m}^{k},\mathbf{n}^{k}
\right],
\mathbf{n}^{l}\mathbf{m}^{i}\mathbf{n}^{i-l-1}
\right]
+
\sum_{k=l+1}^{i-l-1} 
\left[ 
\left[ 
\mathbf{m}^{k},\mathbf{n}^{k}
\right],
\mathbf{n}^{l}\mathbf{m}^{i}\mathbf{n}^{i-l-1}
\right]\nonumber\\
&+& 
\sum_{k=i-l}^{p-1}
\left[ 
\left[ 
\mathbf{m}^{k},\mathbf{n}^{k}
\right],
\mathbf{n}^{l}\mathbf{m}^{i}\mathbf{n}^{i-l-1}
\right].\nonumber
\end{eqnarray}

We notice that the first summation in 
\eqref{eqn:CommutatorHM} 
satisfies the condition of 
Lemma \ref{lem:cases}(a). 
The second summation satisfies the condition of 
Lemma \ref{lem:cases}(b)
and the third summation satisfies the condition of 
Lemma \ref{lem:cases}(d). 
Using Lemma \ref{lem:cases}, we compute the Lie product of 
\eqref{eqn:CommutatorHM} 
as 
\begin{eqnarray*}
\left[
\bar{\mathbf{h}},
\mathbf{n}^{l}\mathbf{m}^{i}\mathbf{n}^{i-l-1}
\right]
&= &
\sum_{k=l+1}^{i-l-1} 
\left(
\mathbf{n}^{k-1}\mathbf{m}^{k+i-l-1}\mathbf{n}^{i-l-1}
-
\mathbf{n}^{k}\mathbf{m}^{k+i-l}\mathbf{n}^{i-l-1}
\right)\\
&&+
\sum_{k=i-l}^{p-1}
\left(
\mathbf{n}^{l}\mathbf{m}^{k+l}\mathbf{n}^{k-1}
-
\mathbf{n}^{k}\mathbf{m}^{k+i-l}\mathbf{n}^{i-l-1}
\right)\\
&&+ 
\sum_{k=i-l}^{p-1}
\left(
\mathbf{n}^{k-1}\mathbf{m}^{i+k-l-1}\mathbf{n}^{i-l-1}
-
\mathbf{n}^{l}\mathbf{m}^{k+l+1}\mathbf{n}^{k}
\right)\\
&=& 
\sum_{k=l+1}^{p-1} 
\left(
\mathbf{n}^{k-1}\mathbf{m}^{k+i-l-1}\mathbf{n}^{i-l-1}
-
\mathbf{n}^{k}\mathbf{m}^{k+i-l}\mathbf{n}^{i-l-1}
\right)\\
&&+
\sum_{k=i-l}^{p-1}
\left(
\mathbf{n}^{l}\mathbf{m}^{k+l}\mathbf{n}^{k-1}
-
\mathbf{n}^{l}\mathbf{m}^{k+l+1}\mathbf{n}^{k}
\right)\\
&=&
\mathbf{n}^{l}\mathbf{m}^{i}\mathbf{n}^{i-l-1}
-
\mathbf{n}^{p-1}\mathbf{m}^{p+i-l-1}\mathbf{n}^{i-l-1}
+
\mathbf{n}^{l}\mathbf{m}^{i}\mathbf{n}^{i-l-1}
-
\mathbf{n}^{l}\mathbf{m}^{p+l}\mathbf{n}^{p-1}\\
&=& 
2\mathbf{n}^{l}\mathbf{m}^{i}\mathbf{n}^{i-l-1}.
\end{eqnarray*}
The last step uses the nilpotency of 
$\mathbf{n}$ and the fact that 
$l\ge 0$ and 
$i-l-1\ge 0$. 
\item[Case II]
Secondly we assume that 
$i-1 < 2l$ and expand the Lie product as 
\begin{eqnarray}
\label{eqn:CommutatorHM2}
[\bar{\mathbf{h}},\mathbf{n}^{l}\mathbf{m}^{i}\mathbf{n}^{i-l-1}]
&=&
\sum_{k=1}^{p-1}
\left[ 
\left[
\mathbf{m}^{k},\mathbf{n}^{k} 
\right],
\mathbf{n}^{l}\mathbf{m}^{i}\mathbf{n}^{i-l-1}
\right]\\
&=& 
\sum_{k=1}^{i-l-1} 
\left[ 
\left[
\mathbf{m}^{k},\mathbf{n}^{k} 
\right],
\mathbf{n}^{l}\mathbf{m}^{i}\mathbf{n}^{i-l-1}
\right]
+
\sum_{k=i-l}^{l} 
\left[ 
\left[
\mathbf{m}^{k},\mathbf{n}^{k} 
\right],
\mathbf{n}^{l}\mathbf{m}^{i}\mathbf{n}^{i-l-1}
\right]\nonumber\\
&+&
\sum_{k=l+1}^{p-1}
\left[ 
\left[
\mathbf{m}^{k},\mathbf{n}^{k} 
\right],
\mathbf{n}^{l}\mathbf{m}^{i}\mathbf{n}^{i-l-1}
\right]\nonumber.
\end{eqnarray}
The first summation in 
\eqref{eqn:CommutatorHM2} 
satisfies the condition of 
Lemma \ref{lem:cases}(a), the second satisfies the condition of 
Lemma \ref{lem:cases}(c), and lastly the third summation satisfies the
condition of  
Lemma \ref{lem:cases}(d). 
Therefore, 
\eqref{eqn:CommutatorHM2} 
can be computed as 
\begin{eqnarray*}
[\bar{\mathbf{h}},\mathbf{n}^{l}\mathbf{m}^{i}\mathbf{n}^{i-l-1}]
&=&
\sum_{k=i-l}^{l} 
\left( 
\mathbf{n}^{l}\mathbf{m}^{k+l}\mathbf{n}^{k-1} 
-
\mathbf{n}^{l}\mathbf{m}^{k+l+1}\mathbf{n}^{k} 
\right)\\
&&+
\sum_{k=l+1}^{p-1}
\left( 
\mathbf{n}^{l}\mathbf{m}^{k+l}\mathbf{n}^{k-1}
-
\mathbf{n}^{k}\mathbf{m}^{k+i-l}\mathbf{n}^{i-l-1}
\right)\\
&&+
\sum_{k=l+1}^{p-1}
\left( 
\mathbf{n}^{k-1}\mathbf{m}^{i+k-l-1}\mathbf{n}^{i-l-1}
-
\mathbf{n}^{l}\mathbf{m}^{k+l+1}\mathbf{n}^{k}
\right)\\
&&= 
\sum_{k=i-l}^{p-1} 
\left( 
\mathbf{n}^{l}\mathbf{m}^{k+l}\mathbf{n}^{k-1} 
-
\mathbf{n}^{l}\mathbf{m}^{k+l+1}\mathbf{n}^{k} 
\right)\\
&&+
\sum_{k=l+1}^{p-1}
\left( 
\mathbf{n}^{k-1}\mathbf{m}^{i+k-l-1}\mathbf{n}^{i-l-1}
-
\mathbf{n}^{k}\mathbf{m}^{k+i-l}\mathbf{n}^{i-l-1}
\right)\\
&=&
\mathbf{n}^{l}\mathbf{m}^{i}\mathbf{n}^{i-l-1} 
-
\mathbf{n}^{l}\mathbf{m}^{p+l}\mathbf{n}^{p-1}
+ 
\mathbf{n}^{l}\mathbf{m}^{i}\mathbf{n}^{i-l-1}
-
\mathbf{n}^{p-1}\mathbf{m}^{p+i-l-1}\mathbf{n}^{i-l-1}\\
&= &
2\mathbf{n}^{l}\mathbf{m}^{i}\mathbf{n}^{i-l-1}.
\end{eqnarray*}
Here the last step follows because of nilpotency of 
$\mathbf{n}$ and 
$\mathbf{m}$.
\item[Case III] 
The last case we treat is when 
$2l=i-1$. 
The Lie product is expanded as 
\begin{eqnarray}
\label{eqn:CommutatorHM3}
[\bar{\mathbf{h}},\bar{\mathbf{m}}]
&=&
\sum_{k=1}^{p-1}
[ 
\mathbf{n}^{k}\mathbf{m}^{k}
- 
\mathbf{m}^{k}\mathbf{n}^{k},
\mathbf{n}^{l}\mathbf{m}^{i}\mathbf{n}^{i-l-1}
]\\
&=& 
\sum_{k=1}^{l} 
[ 
\mathbf{n}^{k}\mathbf{m}^{k}
- 
\mathbf{m}^{k}\mathbf{n}^{k}, 
\mathbf{n}^{l}\mathbf{m}^{i}\mathbf{n}^{i-l-1}
]
+
\sum_{k=l+1}^{p-1} 
[ 
\mathbf{n}^{k}\mathbf{m}^{k}
- 
\mathbf{m}^{k}\mathbf{n}^{k},
\mathbf{n}^{l}\mathbf{m}^{i}\mathbf{n}^{i-l-1}
]\nonumber
\end{eqnarray}
The first summation in 
\eqref{eqn:CommutatorHM3} 
satisfies the conditions of Lemma
\ref{lem:cases}(a) and the second summation satisfies the
conditions of  Lemma
\ref{lem:cases}(d). 
Therefore 
\eqref{eqn:CommutatorHM3} 
can be computed as 
\begin{eqnarray*}
[\bar{\mathbf{h}},\bar{\mathbf{m}}]
&= &
\sum_{k=l+1}^{p-1}
\left( 
\mathbf{n}^{l}\mathbf{m}^{k+l}\mathbf{n}^{k-1}
-
\mathbf{n}^{k}\mathbf{m}^{k+i-l}\mathbf{n}^{i-l-1}
-
\mathbf{n}^{l}\mathbf{m}^{k+l+1}\mathbf{n}^{k}
+ 
\mathbf{n}^{k-1}\mathbf{m}^{i+k-l-1}\mathbf{n}^{i-l-1}
\right)\\
&=& 
\mathbf{n}^{l}\mathbf{m}^{i}\mathbf{n}^{i-l-1} 
-
\mathbf{n}^{l}\mathbf{m}^{pl}\mathbf{n}^{p-1}
-
\mathbf{n}^{p-1}\mathbf{m}^{p+i-l-1}\mathbf{n}^{i-l-1}
+ 
\mathbf{n}^{l}\mathbf{m}^{i}\mathbf{n}^{i-l-1}\\
&=& 
2\mathbf{n}^{l}\mathbf{m}^{i}\mathbf{n}^{i-l-1}.
\end{eqnarray*}
Here the last step follows because of nilpotency of 
$\mathbf{n}$ and 
$\mathbf{m}$.
\end{description}

This concludes the proof.
\end{proof}
\begin{cor}
\label{cor:SL2Mat}
Let
$\mathbf{n}\in \mathfrak{gl}_{n}(\mathbb{R})$ be a nilpotent matrix  with nilpotency index
$p\ge 2$ and let
$\mathbf{m}$ be defined as in Definition \ref{def:m}.
Then the triple
$\langle \subs{\bar{\mathbf{n}}}, \subs{\bar{\mathbf{m}}},\subs{\bar{\mathbf{h}}} \rangle$,
where
\[
\subs{\bar{\mathbf{n}}}
:=
\subs{\mathbf{n}},
\quad
\subs{\bar{\mathbf{m}}}
:=
\sum_{i=1}^{p-1}
\sum_{l=0}^{i-1}
\subs{\mathbf{n}}^{l}\subs{\mathbf{m}}^{i}\subs{\mathbf{n}^{i-l-1}},
\quad
\subs{\bar{\mathbf{h}}}
:=
\sum_{i=1}^{p-1}
\left[
\subs{\mathbf{m}^{i}},
\subs{\mathbf{n}^{i}}
\right],
\]
is an
$\mathfrak{sl}_{2}$-triple.
This follows immediately from Corollary \ref{cor:allrelsstar}.
\end{cor}
\begin{remark}
Notice that \(\subs{\bar{\mathbf{m}}}\) denotes the bar of \(\subs{\mathbf{m}}\), not the star of \(\bar{\mathbf{m}}\).
\end{remark}
\begin{lem}\label{lem:proj}
\(\ker{\nM}=\ker{\bar{\nM}}\)
and
\(\ker\subs{\nM}=\ker\subs{\bar{\nM}}\).
\end{lem}
\begin{proof}
First we notice that \(\nN^{l}\nM^{i}\nN^{i-l-1} = \nN^{l}\nM^{i}\nN^{i-l}\nM\), using the relations. 
Then
\bas
\bar{\nM}&=&\nM+\sum_{i=2}^{p-1} \sum_{l=0}^{i-1}  \nN^{l}\nM^{i}\nN^{i-l-1}
\\&=&\nM+\sum_{i=2}^{p-1} \sum_{l=0}^{i-1}  \nN^{l}\nM^{i}\nN^{i-l}\nM
\\&=&(1+\sum_{i=2}^{p-1} \sum_{l=0}^{i-1}  P_i^l)\nM
\\&=&\mathcal{E}_p \nM,
\eas
where the \(P_i^l\) are as in Lemma \ref{lem:kermm}.
The matrix \(\mathcal{E}_p\) is diagonal, with strictly positive integer entries on the diagonal (since the \(P_i^l\), being projection operators, only contribute zeros or ones on the diagonal), so it is invertible.
It follows that \(\ker \bar{\nM} = \ker \nM\).
To lift this proof to the starred case, we construct a monomial basis \(\langle e_i\rangle_i\) of the space of polynomials in such a way
that \(\subs{\nN}e_i\) is either \(e_{i+1}\) or zero,
and \(\subs{\nN}\) is in Jordan normal form and \(\subs{\nM}\) is its transpose.
\end{proof}
\begin{cor}
\label{cor:proj}
\(\ker{\mathbf{m}}=\ker{\bar{\mathbf{m}}}\)
and
\(\ker\subs{\mathbf{m}}=\ker\subs{\bar{\mathbf{m}}}\).
\end{cor}
\begin{proof}
The \(P\)-conjugate of \(\mathcal{E}_p\) is also invertible.
\end{proof}
\begin{rem}\label{rem:Epp}
The diagonal matrix \(\mathcal{E}_p\) has, in the irreducible case, is given by
\[(1,1\cdot (p-1), 2\cdot (p-2) ,\cdots, (p-1)\cdot 1),\]
 as one would expect from \(\Sl\)-representation theory and as has been shown in Lemma \ref{lem:kermm}.
In the reducible case it consists of blocks of this type.
\end{rem}
\begin{rem}\label{rem:Eppp}
While working on Lemma \ref{lem:kermm}, we found the following interesting identity:
\bas
\nM&=&\bar{\nM}
+\sum_{i=1}^{p-1}(-1)^{i+1}
\sum_{l=1}^{i}
\frac{\binom{i-1}{l-1}\binom{i}{l-1}}{ i! i! l}
 \bar{\nN}^{l-1} 
\bar{\nM}^i 
\bar{\nN}^{i-l}
\\&=&
\bar{\nM}
+\sum_{i=1}^{p-1}(-1)^{i+1}
\sum_{l=1}^{i}
\frac{1}{il}
\binom{i}{l-1}
 \frac{\bar{\nN}^{l-1} }{(l-1)!}
\frac{\bar{\nM}^i }{i!}
\frac{\bar{\nN}^{i-l}}{(i-l)!},
\eas
a formula that we did not prove, but checked in the irreducible case (and some reducible cases) for \(p< 10\).
Since \(i>i-l-1\), \(\bar{\nM}\) can be moved to the right using the \(\Sl\) bracket relations, showing that  \(\ker\bar{\nM}\subset\ker\nM\).
\end{rem}
\begin{thm}
Let \(\cann{\bar{\mathbf{h}}}=\mult{\bar{\mathbf{h}}}+\subs{\bar{\mathbf{h}}}\). Then \(\langle \conn{\bar{\mathbf{m}}},\cann{\bar{\mathbf{h}}},\conn{\bar{\mathbf{n}}}\rangle\)
is an
$\mathfrak{sl}_{2}$-triple.
\end{thm}
\begin{rem}
This has as a consequence that \(\im\conn{\mathbf{n}}\) has  \(\ker\conn{\bar{\mathbf{m}}}\) as a natural complement.
\end{rem}
\begin{proof}
We compute
\bas
\left[ \conn{\bar{\mathbf{m}}},\conn{\bar{\mathbf{n}}}\right]&=& \conn{\bar{\mathbf{m}}}\conn{\bar{\mathbf{n}}}-\conn{\bar{\mathbf{n}}}\conn{\bar{\mathbf{m}}}
\\&=& (\mult{\bar{\mathbf{m}}}-\subs{\bar{\mathbf{m}}})(\mult{\bar{\mathbf{n}}}-\subs{\bar{\mathbf{n}}})-(\mult{\bar{\mathbf{n}}}-\subs{\bar{\mathbf{n}}})(\mult{\bar{\mathbf{m}}}-\subs{\bar{\mathbf{m}}})
\\&=&\mult{\bar{\mathbf{m}}} \mult{\bar{\mathbf{n}}} -\subs{\bar{\mathbf{m}}}\mult{\bar{\mathbf{n}}} -\mult{\bar{\mathbf{m}}}\subs{\bar{\mathbf{n}}}+\subs{\bar{\mathbf{m}}}\subs{\bar{\mathbf{n}}}
\\&+&\mult{\bar{\mathbf{n}}}\subs{\bar{\mathbf{m}}}-\mult{\bar{\mathbf{n}}}\mult{\bar{\mathbf{m}}}+\subs{\bar{\mathbf{n}}}\mult{\bar{\mathbf{m}}}-\subs{\bar{\mathbf{n}}}\subs{\bar{\mathbf{m}}}
\\&=&\left[\mult{\bar{\mathbf{m}}} , \mult{\bar{\mathbf{n}}}\right]+\left[\subs{\bar{\mathbf{m}}},\subs{\bar{\mathbf{n}}}\right]
\\&=&\mult{\bar{\mathbf{h}}}+\subs{\bar{\mathbf{h}}} =\cann{\bar{\mathbf{h}}},
\eas
\bas
\left[ \cann{\bar{\mathbf{h}}},\conn{\bar{\mathbf{m}}}\right]&=&
\cann{\bar{\mathbf{h}}}\conn{\bar{\mathbf{m}}}-\conn{\bar{\mathbf{m}}}\cann{\bar{\mathbf{h}}}
\\&=&(\mult{\bar{\mathbf{h}}}+\subs{\bar{\mathbf{h}}})(\mult{\bar{\mathbf{m}}}-\subs{\bar{\mathbf{m}}})-(\mult{\bar{\mathbf{m}}}-\subs{\bar{\mathbf{m}}})(\mult{\bar{\mathbf{h}}}+\subs{\bar{\mathbf{h}}})
\\&=&\mult{\bar{\mathbf{h}}} \mult{\bar{\mathbf{m}}} -\mult{\bar{\mathbf{h}}}\subs{\bar{\mathbf{m}}}+\subs{\bar{\mathbf{h}}}\mult{\bar{\mathbf{m}}}-\subs{\bar{\mathbf{h}}}\subs{\bar{\mathbf{m}}}
\\&&-\mult{\bar{\mathbf{m}}}\mult{\bar{\mathbf{h}}}-\mult{\bar{\mathbf{m}}}\subs{\bar{\mathbf{h}}}+\subs{\bar{\mathbf{m}}}\mult{\bar{\mathbf{h}}}+\subs{\bar{\mathbf{m}}}\subs{\bar{\mathbf{h}}}
\\&=&\left[\mult{\bar{\mathbf{h}}},\mult{\bar{\mathbf{m}}}\right]-\left[\subs{\bar{\mathbf{h}}},\subs{\bar{\mathbf{m}}}\right]=2\mult{\bar{\mathbf{m}}}-2\subs{\bar{\mathbf{m}}}=2\conn{\bar{\mathbf{m}}},
\eas
\bas
\left[ \cann{\bar{\mathbf{h}}},\conn{\bar{\mathbf{n}}}\right]&=&
\cann{\bar{\mathbf{h}}}\conn{\bar{\mathbf{n}}}-\conn{\bar{\mathbf{n}}}\cann{\bar{\mathbf{h}}}
\\&=&(\mult{\bar{\mathbf{h}}}+\subs{\bar{\mathbf{h}}})(\mult{\bar{\mathbf{n}}}-\subs{\bar{\mathbf{n}}})-(\mult{\bar{\mathbf{n}}}-\subs{\bar{\mathbf{n}}})(\mult{\bar{\mathbf{h}}}+\subs{\bar{\mathbf{h}}})
\\&=&\mult{\bar{\mathbf{h}}} \mult{\bar{\mathbf{n}}} -\mult{\bar{\mathbf{h}}}\subs{\bar{\mathbf{n}}}+\subs{\bar{\mathbf{h}}}\mult{\bar{\mathbf{n}}}-\subs{\bar{\mathbf{h}}}\subs{\bar{\mathbf{n}}}
\\&&-\mult{\bar{\mathbf{n}}}\mult{\bar{\mathbf{h}}}-\mult{\bar{\mathbf{n}}}\subs{\bar{\mathbf{h}}}+\subs{\bar{\mathbf{n}}}\mult{\bar{\mathbf{h}}}+\subs{\bar{\mathbf{n}}}\subs{\bar{\mathbf{h}}}
\\&=&\left[\mult{\bar{\mathbf{h}}},\mult{\bar{\mathbf{n}}}\right]-\left[\subs{\bar{\mathbf{h}}},\subs{\bar{\mathbf{n}}}\right]=-2\mult{\bar{\mathbf{n}}}+2\subs{\bar{\mathbf{n}}}=-2\conn{\bar{\mathbf{n}}}.
\eas
This proves the main result for this paper.
\end{proof}
The remaining sections will apply this result to a number of cases.
\section{The normal form in the irreducible nilpotent case}\label{sec:NFirreducible}
Assume \(\nN\) is an irreducible nilpotent matrix in Jordan normal form in \(\mathbb{R}^n\) and let \(\nM\) be its transpose.
Then \(\subs{\nN}\) shifts the indices in a monomial up by one, and
 \(\subs{\nM}\) shifts the indices in a monomial down by one.
We write \(|k,l|\) (with \(k\le l\)) for an element \(x_1^{i_1} x_2^{i_2} \cdots x_n^{i_n}\) with \(i_k,i_l>0\) and \(i_1,\ldots,i_{k-1},i_{l+1},\ldots,i_n\) are zero.
That is, we mention only the lowest and the highest index occuring in the monomial.
In other words, it is a monomial of the form \(x_kx_l F(x_k,\ldots,x_l)\).
\begin{rem}
From the eigenvalue point of view it would be more natural to use \(y_{n-1},\cdots,y_0\) instead of \(x_1,\cdots,x_n\) as coordinates.
\end{rem}
Then we compute as follows:
\(\subs{\nN}|k,l|=|\nN(k,l)|=|k+1,l+1|\) if \(l<n\) and  \(0\) if \(l=n\).
The kernel of \(\subs{\nM} \) (and of \(\subs{\bar{\nM}} \)) is spanned by \(|1,k|\), \(k=1,\ldots,n\).
\bas
\ker\subs{\nM}&=&
\bigoplus_{l=0}^{n-2}u^{l} x_1 x_{n-l} \mathbb{R}[u^{n-1} x_1,\cdots,u^{l} x_{n-l}]
\oplus u^{n-1} x_1 \mathbb{R}[u^{n-1} x_1].
\eas
The \(\subs{\bar{\nH}}\)-eigenvalue of \(|1,k|\) is \(n-k\).
Thus we have a decomposition \[\mathcal{P}^n=(V_{n-1} \oplus V_{n-2} \oplus\ldots\oplus V_0)\otimes V_{n-1},\]
generated by \(|1,1||n\rangle,|1,2||n\rangle,\cdots,|1,n||n\rangle\).
Since \(V_{k-1}\otimes V_{n-1}=\bigoplus_{l=n-k}^{n +k-2}V_l\) by Clebsch-Gordan, we have
\[
\mathcal{P}^n=\bigoplus_{k=1}^n V_{k-1}\otimes V_{n-1}= \bigoplus_{k=1}^n \bigoplus_{l=n-k}^{n +k-2}V_l,
\]
where the elements in \(V_l\) have \(\cann{\nH}\)-eigenvalue \(l\).
 \begin{defn}
We define the \(p\)th {\em transvectant} of \(\nw_{k-1} \otimes \nv_{n-1}\), where  \(\nw_{k-1}\) is a top weight vector with eigenvalue \(k-1\) and \(\nv_{n-1}\) a top weight vector with eigenvalue \(n-1\),
(think of \(|1,n-k+1|\) and \(|n\rangle\), respectively), \[\Join_{k+n-2p-2}: W_{k-1}\otimes V_{n-1} \rightarrow U_{k+n-2p-2}\cap \ker \conn{\bar{\nM}},\] as
\bas
\Join_{k+n-2p-2}\nw_{k-1}\otimes\nv_{n-1} =
\sum_{i+j=p}  \binom{p}{i}  \frac{\nw^{(j)}_{k-1}}{\binom{k-1}{j}}\otimes \frac{\nv^{(i)}_{n-1}}{\binom{n-1}{i}},\quad k=1,\ldots,n, p=0,\ldots,\min(k-1,n-1),
\eas
where \(\nw^{(j)}_{k-1}=\frac{1}{j!}\subs{\nN}^j \nw_{k-1}\) and \(\nv^{(i)}_{n-1}=\frac{1}{i!}\mult{\nN}^i \nv_{n-1}\), \(0\le i\leq n-1\) and \(0\le j\le k-1\).
\end{defn}
\begin{rem}
The transvectant was one of the main tools of Classical Invariant Theory to compute covariants, that is, irreducible \(\Sl\)-representations.
In the modern theory it is replaced by the much more general Young-Weyl tableaux to cover more general Lie algebras.
But in normal form theory there is no need for such generality (yet), and we can enjoy working with explicit formulas to implement the Clebsch-Gordan decomposition.
\end{rem}
\begin{rem}
In the usual transvectant definition there is always a factor \((-1)^j\); here we do not have this sign appearing because it is present in the definition of \(\conn{\nN}\).
\end{rem}
In our example 
\[\nw^{(j)}_{k-1}=\frac{1}{j!}\subs{\nN}^j |1,n-k+1|=\frac{1}{j!}|j+1,n-k+j+1|\]
and 
\[\nv^{(i)}_{n-1}=\frac{1}{i!}\mult{\nN}^i |n\rangle=\frac{1}{i!}|n-i\rangle.\] 
The transvectant now reads
\ba\label{eq:transvectant}
\Join_{k+n-2p-2}|1,n-k+1||n\rangle =
\sum_{i+j=p} \frac{1}{i!j!}\frac{ \binom{p}{i} }{\binom{n-1}{i} \binom{k-1}{j}} |j+1,n-k+j+1||n-i\rangle.
\ea
Here we can replace \(|i+1,n-k+i+1|\) by \(x_{i+1}x_{n-k+i+1}F_{n-k+1}^p(x_{i+1},\cdots,x_{n-k+i+1})\) 
where \(F_{n-k+1}^{p}\) is an arbitrary polynomial and
\(|n-i\rangle\) by \(e_{n-i}\) to obtain the corresponding term in the general normal form formula.

\begin{defn}
Define the Clebsch-Gordan Coefficient, also known as \(3j\)-symbol , by (cf. \cite{mokhtari2019equivariant})
\[\begin{pmatrix} m& n& m+n-2p\\i&j& k\end{pmatrix}=\sum_{r+q=k} (-1)^{i-k+r} \frac{\binom{p}{i-q}\binom{i}{q} \binom{j}{r}}{\binom{m}{i-q}\binom{n}{j-r}}, \quad i+j=k+p,\]
with \(j_1=m, j_2=n\) and \(j_3=n+m-2p\) as the three \(j\)'s. 
Here all the \(n\) and \(k\) are arbitrary and not determined by their previous meaning.
\end{defn}
\begin{rem}
	We consider the special case \(k=0\):
\[\begin{pmatrix} m& n& m+n-2p\\i&j& 0\end{pmatrix}= \frac{\binom{p}{i} }{\binom{m}{i}\binom{n}{j}}, \quad i+j=p.\]	
\end{rem}
Now recall the inversion formula for the Clebsch-Gordan coefficients \cite{mokhtari2019equivariant}:
\bas
\nw^{(j)}_n\otimes \nv^{(i)}_m&=&\sum_{p+k=i+j}\begin{pmatrix} m&n&m+n-2p\\i&j&k\end{pmatrix}\frac{ \binom{m}{i}\binom{n}{j}\binom{m}{p}\binom{n}{p}}{\binom{m+n-2p}{k}\binom{m+n-p+1}{p}}
\Join_{m+n-2p}^{(k)}\nw_n\otimes \nv_m,
\eas
with the proviso that \(p\le \min(m,n)\).
This allows us to project an arbitrary expression onto \(\ker \conn{\bar{\nM}}\) by taking \(k=0\):
\bas
\pi_{\ker \conn{\bar{\nM}}}\nw^{(j)}_n\otimes\nv^{(i)}_m&=&\sum_{p=i+j}\begin{pmatrix} m&n&m+n-2p\\i&j&0\end{pmatrix}\frac{ \binom{m}{i}\binom{n}{j}\binom{m}{p}\binom{n}{p}}{\binom{m+n-p+1}{p}}
\Join_{m+n-2p}\nw_n\otimes \nv_m
\\
&=&\sum_{p=i+j}\frac{ \binom{p}{i}\binom{m}{p}\binom{n}{p}}{\binom{m+n-p+1}{p}}
\Join_{m+n-2p}\nw_n\otimes\nv_m
\\
&=&\frac{ \binom{i+j}{i}\binom{m}{i+j}\binom{n}{i+j}}{\binom{m+n-(i+j)+1}{i+j}}
\Join_{m+n-2(i+j)}\nw_n\otimes\nv_m.
\eas
\begin{rem}
	Notice that the projection on \(\ker \conn{\bar{\nM}}\) is done without using \(\bar{\nM}\): only \(\nN\) is used.
	The Clebsch-Gordan inversion formula not only allows us to project on \(\ker\conn{\bar{\nM}}\),
	but also to find the preimage of \(\ker\conn{\bar{\nN}}\) by lowering \(k\) to \(k-1\) and dividing by \(k\).
	Again,  only \(\nN\) is used.
\end{rem}
In our example \[\nv^{(i)}_{n-1}=\frac{1}{i!}\mult{\nN}^i |n\rangle=\frac{1}{i!}|n-i\rangle\] and \[\nw^{(j)}_{k-1}=\frac{1}{j!}\subs{\nN}^j |1,n-k+1|=\frac{1}{j!}|j+1,n-k+j+1|.\]
It follows that
\bas
\pi_{\ker \conn{\bar{\nM}}}|j+1,n-k+j+1||n-i\rangle&=&
\frac{ (i+j)!\binom{{n-1}}{i+j}\binom{{k-1}}{i+j}}{\binom{{n}+{k}-(i+j)-1}{i+j}}
\Join_{n+k-2-2(i+j)} |1,n-k+1||n\rangle.
\eas
or
\ba
\label{eq:projection}
\pi_{\ker \conn{\bar{\nM}}} |I,J| |K\rangle&=&\frac{(n-1)!}{(K-I)!}
\frac{\binom{n-J+I-1}{n-K+I-1}}{\binom{n-J+K}{n-K+I-1}}
\Join_{2K-J-I} |1,J-I+1||n\rangle.
\ea
The transvectant now reads
\[
\Join_{k+n-2p-2}|1,n-k+1||n\rangle =
\sum_{i+j=p} \frac{1}{i!j!}\frac{ \binom{p}{i} }{\binom{n-1}{i} \binom{k-1}{j}} |j+1,n-k+j+1||n-i\rangle.
\]
This allows us to compute the general form of the normal form (the description problem) and to do the actual computations for any dimension, only using \(\nN\)
and the knowledge of \(\ker{\nM} \) and \(\ker\subs{{\nM}}\) (thanks to Lemma \ref{lem:kermm}).
\subsection{The normal form of the two dimensional nilpotent map}
The transvectant in \eqref{eq:transvectant} now reads, with \(n=2\),
\[
\Join_{k-2p}|1,3-k||2\rangle =
\sum_{i+j=p} \frac{1}{i!j!}\frac{ \binom{p}{i} }{\binom{1}{i} \binom{k-1}{j}} |j+1,3-k+j||2-i\rangle,\quad k=1,2, p=0,k-1.
\]
Here we can replace \(|j+1,3-k+j|\) by \(x_{j+1}x_{3-k+j}F_{3-k}^p(x_{j+1},\cdots,x_{3-k+j})\) 
where \(F_{3-k}^{p}\) is an arbitrary polynomial and
\(|2-i\rangle\) by \(e_{2-i}\) to obtain the corresponding term in the general normal form formula.

It follows from the Clebsch-Gordan formula that \(0\le p\le k-1\).
\newline
For \(k=1\) and \(p=0\) we have
\[
\Join_{1}|1,2||2\rangle =
|1,2||2\rangle,
\]
corresponding to
\[
\begin{pmatrix} 0\\ x_1 x_2 F_{2}^0 (x_1,x_2)\end{pmatrix}.
\]
For \(k=2\) and \(p=0\) we have
\[
\Join_{2}|1,1||2\rangle =
|1,1||2\rangle,
\]
corresponding to
\[
\begin{pmatrix} 0\\ x_1 F_{1}^0 (x_1)\end{pmatrix}.
\]
For \(k=2\) and \(p=1\) we have
\[
\Join_{0}|1,1||2\rangle =
|j+1,1+j||2-i\rangle=|2,2||2\rangle
+ |1,1||1\rangle,
\]
corresponding to
\[
\begin{pmatrix} x_1 F_1^1(x_1) \\ x_2 F_{1}^1 (x_2)\end{pmatrix}.
\]
The general normal form map in \(\ker \conn{\bar{\nM}}\)-style is now
\[
F(x)
=
\begin{pmatrix}
0 & 1 \\
0 & 0
\end{pmatrix}
\begin{pmatrix}
x_{1}\\
x_{2}
\end{pmatrix}
+
\begin{pmatrix} 0\\ x_1 x_2 F_{2}^0 (x_1,x_2)\end{pmatrix}
+
\begin{pmatrix} 0\\ x_1 F_{1}^0 (x_1)\end{pmatrix}
+
\begin{pmatrix} x_1 F_1^1(x_1) \\ x_2 F_{1}^1 (x_2)\end{pmatrix}.
\]
Here we can combine the \( F_{2}^0\) and \(F_{1}^0 \) to
\[
\begin{pmatrix} 0\\ x_1 p (x_1,x_2)\end{pmatrix}
\]
to get a shorter expression (but which is less natural with respect to \(\Sl\)).
Notice that the versal deformation of the linear part is given by
\bas
\begin{pmatrix}
\alpha_1^1 & 1 \\
\alpha_1^0 & \alpha_1^1
\end{pmatrix},
\eas
where the indices of \(\alpha\) correspond to those of \(F\). This is the same as in the vector field case since the map is linear.
\begin{rem}
In \cite[Example 2.4]{MR1735239}, the authors compute the normal form of this
map, without the use of
$\mathfrak{sl}_{2}$ theory: 
\[
\widetilde{f}(x)
=
\begin{pmatrix}
0 & 1\\
0 & 0
\end{pmatrix}
\begin{pmatrix}
x_{1}\\
x_{2}
\end{pmatrix}
+
\sum_{n=2}^{\infty}
\sum_{k=0}^{n}
\alpha_{n,k}
\begin{pmatrix}
0 \\
x_{1}^{k}x_{2}^{n-k}
\end{pmatrix}
.
\]
One could call this the normal form in \(\ker \mult{\nM}\)-style, since the normal form terms are in \(\mathbb{R}[x_1,x_2]\otimes \ker{\nM}\).
Is this style in general equivalent to the \(\ker \conn{\bar{\nM}}\)-style? If so, it would certainly simplify the description problem of the general nilpotent normal form.
We will try to shed some light on the answer to this question in the following.
\end{rem}
If we apply the projection formula \eqref{eq:projection} on \(\ker \conn{\bar{\nM}}\) to the term \(\begin{pmatrix}
0 \\
x_2 F_1^1(x_{2})
\end{pmatrix}=|2,2||2\rangle
\)
we obtain
\bas
\pi_{\ker \conn{\bar{\nM}}} |2,2| |2\rangle&=&
\frac{1}{2}
\Join_{0} |1,1||2\rangle
\\&=&
\frac{1}{2} (|1,1||1\rangle+|2,2|2\rangle)=\frac{1}{2}\begin{pmatrix}
x_1 F_1^1(x_1 ) \\
x_2 F_1^1(x_2)
\end{pmatrix}
.
\eas
This leads us to the conclusion that for \(n=2\) we can identify \(\ker \conn{\bar{\nM}}\) with \(\mathbb{R}[x_1,x_2]\otimes\ker{\nM}\).
This last formulation of the normal form style has an interesting computational aspect: in order to find a normal form transformation it is enough to solve 
\(\varphi\) from the equation
\[
\mult{\nM}(\conn{\nN}\varphi(x)-F(x))=0
\]
at some fixed degree,
which looks a lot better than solving from
\[
\conn{\bar{\nM}}(\conn{\nN}\varphi(x)-F(x))=0,
\]
which  is the usual computational method.

We may conjecture that this holds in general, i.e. \(\ker \conn{\bar{\nM}}\equiv  \mathbb{R}[x_1,\ldots,x_n]\otimes  \ker{\nM}\).
Indeed, if we project it on \(\ker\mult{\nM}\) we find
\[
\pi_{\ker\mult{\nM}}\Join_{2n-2I+1-J}|1,J||n\rangle =
 \frac{(n-J-I+1)!}{(n-J)!} |I,I+J-1||n\rangle, \quad J=1,\ldots,n.
\]
This shows that \(\ker\mult{\nM}\equiv \ker \conn{\bar{\nM}}\) for irreducible \(\nN\).
We will show in Section \ref{sec:reducible} that this no longer holds in general for the reducible case, cf. Remark \ref{rem:kerm}.
\subsubsection{Example of a normal form calculation}
We give a very simple example, just to give the reader an impression how this might work in general.
The reader should keep in mind that this is more suitable for an automated calculation and not very convincing as a hand calculation.
Consider the map
\bas
 \begin{pmatrix} x_1\\ x_2\end{pmatrix}\mapsto \begin{pmatrix} x_2 + a_{11} x_1^2 +2 a_{12} x_1 x_2 +a_{22} x_2^2\\
b_{11} x_1^2 +2 b_{12} x_1 x_2 +b_{22} x_2^2
\end{pmatrix}.
\eas
We see that \(\nN\) acts as \(\nN |1\rangle=0\) and \(\nN |2\rangle=|1\rangle\).
And \(\nM\) acts as \(\nM |1\rangle=|2\rangle\) and \(\nM |2\rangle=0\).
Thus \(\ker\nM\) is spanned by \(|2\rangle\).
Furthermore \(\subs{\nN} |1,1|=|2,2|\), \(\subs{\nN} |1,2|=0\) and \(\subs{\nN} |2,2|=0\).
And \(\subs{\nM} |1,1|=0\), \(\subs{\nM} |1,2|=0\) and \(\subs{\nM} |2,2|=|1,1|\).
Thus \(\ker\subs{\nM}\) is spanned by \(|1,1|,|1,2|\).
In our example 
\bas
\nv^{(0)}_{1}&=&|2\rangle,\quad
\nv^{(1)}_{1}=\mult{\nN} |2\rangle=|1\rangle.
\eas
and 
\bas
\nw^{(0)}_{0}&=& |1,2|,\quad
\nw^{(0)}_{1}= |1,1|,\quad
\nw^{(1)}_{1}=\subs{\nN} |1,1|=|2,2|.
\eas
Now recall the inversion formula for the Clebsch-Gordan coefficients \cite{mokhtari2019equivariant}:
\bas
|1,2||2\rangle=\nw^{(0)}_{0}\otimes\nv^{(0)}_{1}&=&\begin{pmatrix} 1&0&1\\0&0&0\end{pmatrix}
\Join_{1}^{(0)}\nw_0\otimes\nv_1=\nw_0\otimes\nv_1,
\\
|1,2||1\rangle=\nw^{(0)}_{0}\otimes\nv^{(1)}_{1}&=&
\begin{pmatrix} 1&0&1\\1&0&1\end{pmatrix} \Join_{1}^{(1)}\nw_0\otimes\nv_1=\Join_{1}^{(1)}\nw_0\otimes\nv_1=\conn{\nN}\nw_0\otimes\nv_1,
\\
|1,1||2\rangle=\nw^{(0)}_{1}\otimes\nv^{(0)}_{1}&=&
\begin{pmatrix} 1&1&2\\0&0&0\end{pmatrix}
\Join_{2}^{(0)}\nw_1\otimes\nv_1
=\nw_1\otimes \nv_1,
\\
|1,1||1\rangle=\nw^{(0)}_{1}\otimes \nv^{(1)}_{1}&=&
\frac{1}{2}\begin{pmatrix} 1&1&2\\1&0&1\end{pmatrix} \Join_{2}^{(1)}\nv_1\otimes \nw_1\otimes\nv_1
+\frac{1}{2}\begin{pmatrix} 1&1&0\\1&0&0\end{pmatrix} \Join_{0}^{(0)}\nw_1\otimes\nv_1
\\&=&
\frac{1}{2} \conn{\nN}\nw_1\otimes\nv_1
-\frac{1}{2} \Join_{0}^{(0)}\nw_1\otimes\nv_1,
\\
|2,2||2\rangle=\nw^{(1)}_{1}\otimes\nv^{(0)}_{1}&=&
\frac{1}{2}\begin{pmatrix} 1&1&2\\0&1&1\end{pmatrix} \Join_{2}^{(1)} \nw_1\otimes\nv_1
+\begin{pmatrix} 1&1&0\\0&1&0\end{pmatrix} \Join_{0}^{(0)}\nw_1\otimes\nv_1
\\&=&
\frac{1}{2}\conn{\nN} \nw_1\otimes\nv_1
+ \Join_{0}^{(0)}\nw_1\otimes\nv_1,
\\
|2,2||1\rangle=\nw^{(1)}_{1}\otimes\nv^{(1)}_{1}&=&
\begin{pmatrix} 1&1&2\\1&1&2\end{pmatrix}\Join_{2}^{(2)}\nw_1\otimes\nv_1
=\frac{1}{2}\conn{\nN} \Join_{2}^{(1)}\nw_1\otimes\nv_1.
\eas
It follows that
\bas
&&+a_{11} |1,1||1\rangle
+ \frac{1}{2} a_{12} |1,2||1\rangle
+a_{22} |2,2||1\rangle 
+
b_{11} |1,1||2\rangle
+ \frac{1}{2} b_{12} |1,2||2\rangle
+b_{22} |2,2||2\rangle 
\\&=&
-\frac{1}{2}a_{11}  \Join_{0}\nw_1\otimes\nv_1
+b_{11} \nw_1\otimes\nv_1
+ \frac{1}{2} b_{12} \nw_0\otimes\nv_1
+b_{22} \Join_{0}\nw_1\otimes\nv_1
\\&&
+\frac{1}{2}\conn{\nN}(a_{11}  \nw_1\otimes\nv_1
+  a_{12} \nw_0\otimes\nv_1
+ a_{22}  \Join_{2}^{(1)}\nw_1\otimes\nv_1
+b_{22} \nw_1\otimes\nv_1).
\eas
The first order normal form is now
\bas
 \begin{pmatrix} x_1\\ x_2\end{pmatrix}\mapsto \begin{pmatrix} x_2\\0\end{pmatrix} +(b_{22} - \frac{1}{2}  a_{11})  \begin{pmatrix} x_1^2 \\x_2^2\end{pmatrix}
+b_{11}  \begin{pmatrix} 0\\x_1^2\end{pmatrix} 
+\frac{1}{2} b_{12}  \begin{pmatrix} 0\\x_1 x_2 \end{pmatrix}.
\eas
We see that the number of terms is half of that of the original map. In the \(n\)-dimensional case one would expect \({1}/{n}\) terms in the normal form, at least in the irreducible case.
\subsection{The three dimensional nilpotent normal form}
The transvectant in \eqref{eq:transvectant} now reads, with \(n=3\),
\[
\Join_{k-2p+1}|1,4-k||3\rangle =
\sum_{i+j=p} \frac{1}{i!j!}\frac{ \binom{p}{i} }{\binom{2}{i} \binom{k-1}{j}} |j+1,4-k+j||3-i\rangle, k=1,2,3, p=0,\ldots,k-1.
\]
We now do \(k=1,\ldots,3\) and \(p=0,\ldots,k-1\) (playing computer):
\newline
For \(k=1\) and \(p=0\) we have
\[
\Join_{2}|1,3||3\rangle =
 |1,3||3\rangle,
\]
corresponding to
\[
\begin{pmatrix} 0\\0\\x_1 x_2 x_3 F_3^0(x_1,x_2,x_3) \end{pmatrix}.
\]
For \(k=2\) and \(p=0\) we have
\[
\Join_{3}|1,2||3\rangle =
 |1,2||3\rangle,
\]
corresponding to
\[
\begin{pmatrix} 0\\0\\x_1 x_2  F_2^0(x_1,x_2) \end{pmatrix}.
\]
For \(k=3\) and \(p=0\) we have
\[
\Join_{4}|1,1||3\rangle =
 |1,1||3\rangle,
\]
corresponding to
\[
\begin{pmatrix} 0\\0\\x_1  F_1^0(x_1) \end{pmatrix}.
\]
For \(k=2\) and \(p=1\) we have
\[
\Join_{1}|1,2||3\rangle =
\sum_{i+j=1} \frac{ 1 }{\binom{2}{i} } |j+1,4-k+j||3-i\rangle= |2,3||3\rangle
+\frac{ 1 }{2 } |1,2||2\rangle,
\]
corresponding to
\[
\begin{pmatrix} 0\\\frac{1}{2} x_1 x_2 F_2^1(x_1,x_2) \\x_2 x_3 F_2^1(x_2,x_3) \end{pmatrix}.
\]
For \(k=3\) and \(p=1\) we have
\[
\Join_{2}|1,1||3\rangle =
\sum_{i+j=1} \frac{ 1}{\binom{2}{i} \binom{2}{j}} |j+1,4-k+j||3-i\rangle= \frac{1}{2} |2,2||3\rangle+\frac{ 1}{2}  |1,1||2\rangle,
\]
corresponding to
\[
\frac{1}{2} \begin{pmatrix} 0\\ x_1 F_1^1(x_1) \\x_2 F_1^1(x_2) \end{pmatrix}.
\]
For \(k=3\) and \(p=2\) we have
\[
\Join_{0}|1,1||3\rangle =
\sum_{i+j=2} \frac{1}{i!j!}\frac{ 1 }{ \binom{2}{j}} |j+1,4-k+j||3-i\rangle=
\frac{1}{2} |3,3||3\rangle
+\frac{ 1 }{ 2} |2,2||2\rangle
+\frac{1}{2} |1,1||1\rangle,
\]
corresponding to
\[
\frac{1}{2} \begin{pmatrix} x_1 F_1^2(x_1)\\ x_2 F_1^2(x_2) \\x_3 F_1^2(x_3) \end{pmatrix}.
\]
The normal form is now (where we ignore the superfluous \(\frac{1}{2}\)s):
\bas
F(x)
&= &
\begin{pmatrix}
0 & 1 & 0\\
0 & 0 & 1\\
0 & 0 & 0 
\end{pmatrix}
\begin{pmatrix}
x_{1}\\
x_{2}\\
x_{3}
\end{pmatrix}
+
\begin{pmatrix} 0\\0\\x_1 x_2 x_3 F_3^0(x_1,x_2,x_3) \end{pmatrix}
+
\begin{pmatrix} 0\\0\\x_1 x_2  F_2^0(x_1,x_2) \end{pmatrix}
\\&+&
\begin{pmatrix} 0\\0\\x_1  F_1^0(x_1) \end{pmatrix}
+
\begin{pmatrix} 0\\\frac{1}{2} x_1 x_2 F_2^1(x_1,x_2) \\x_2 x_3 F_2^1(x_2,x_3) \end{pmatrix}
+
\begin{pmatrix} 0\\ x_1 F_1^1(x_1) \\x_2 F_1^1(x_2) \end{pmatrix}
+
\begin{pmatrix} x_1 F_1^2(x_1)\\ x_2 F_1^2(x_2) \\x_3 F_1^2(x_3) \end{pmatrix}.
\eas
Notice that the versal deformation of the linear part is given by
\bas
\begin{pmatrix}
\alpha_1^2 & 1 & 0\\
\alpha_1^1 & \alpha_1^2 & 1\\
\alpha_1^0 & \alpha_1^1 & \alpha_1^2 
\end{pmatrix},
\eas
where the indices of \(\alpha\) correspond to those of \(F\). This is the same as in the vector field case since the map is linear.
\section{The reducible case}\label{sec:reducible}
The main tool we have been using so far, is the Clebsch-Gordan formula and its inversion formula.
This will work equally well in the reducible case, but the bookkeeping will get more complicated.
We therefore restrict ourself to case with two blocks but first give several examples and formulate a conjecture in Remark \ref{rem:kerm}, before proving the conjecture in the case of two blocks in Section \ref{sec:2blocks}.


We take here a direct approach to the computation of the normal form. 
In \cite[Chapter 12]{sanders2007averaging} it is shown how nilpotent normal forms of vector fields can be `added'. 
This may also be a viable approach for maps: the tensor products involved are much simpler here.
The main issue is to avoid relations in the normal form expression and that is exactly what the Clebsch-Gordan formula delivers.
\subsection{The $(2,3)$-reducible case}
In Theorem \ref{thm:reductive23} we give the explicit matrix \(\nN\).

In order to construct a general formula for the elements in  \(\ker\subs{\nM}\) we proceed as follows.
First we collect the generators of the kernel: \( u x_1,u^2 x_3,x_1x_2,x_1x_5, x_2x_3,u x_3x_4,x_3x_5\). These we multiply with \(\mathbb{R}[u x_1,x_2,u^2 x_3,u x_4,x_5]\).
Then we look for linear dependencies and if necessary, adapt the multiplying polynomial ring.
This corresponds 
to a decomposition of \(\ker\subs{\nM}\)
of type (with the \(u\) powers according to the \(\subs{\mathbf{h}}\)-eigenvalues)
\bas
\ker\subs{\nM}&=&x_1 x_2 \mathbb{R}[u x_1,x_2,u^2 x_3,u x_4,x_5]
\oplus
x_1 x_5 \mathbb{R}[u x_1,u^2x_3,ux_4,x_5]
\\&\oplus&
x_2 x_3 \mathbb{R}[x_2,u^2 x_3,u x_4,x_5]
\oplus
x_3 x_5 \mathbb{R}[u^2 x_3,u x_4,x_5]
\\&\oplus&u(x_1 \mathbb{R}[u x_1,u^2 x_3,u x_4]
\oplus x_3 x_4\mathbb{R}[u^2 x_3,u x_4])
\oplus u^2 x_3 \mathbb{R}[u^2 x_3].
\eas
This decomposition is computed as follows: we start with \(x_1 x_2 \mathbb{R}[ux_1,x_2,u^2 x_3,ux_4,x_5]\) and
define the set \(\mathcal{S}_1=\left\{ x_1 x_2, x_1 x_2 x_3,x_1 x_2 x_4,x_1 x_2 x_5\right\}\).
Then we consider \(x_1 x_5 \mathbb{R}[u x_1,x_2,u^2 x_3,u x_4,x_5]\) and construct the set \(\mathcal{E}_1=\left\{ x_1 x_5, x_1 x_2 x_5, x_1 x_3 x_5 ,x_1 x_4 x_5 x_4\right\}\).
We then take the intersection of \(\mathcal{S}_1\) and \(\mathcal{E}_1\), that is \(\mathcal{I}_1=\left\{ x_1 x_2 x_5\right\}\).
If this intersection is nonempty, we remove the variable that causes this from the list. In this case the problem is caused by \(x_2\) and so we continue with
\(x_1 x_5 \mathbb{R}[u x_1,u^2 x_3,u x_4,x_5]\) and put \(\mathcal{S}_2=\mathcal{S}_1\bigcup (\mathcal{E}_1\setminus \mathcal{I}_1)\).
	We then proceed with the next term \(x_2x_3\), \&c.

\begin{rem}
An alternative would be to compute the tensor product of \(\ker\subs{\nM}|\mathbb{R}[u x_1,x_2]\) and \(\ker\subs{\nM}|\mathbb{R}[u^2 x_3,u x_4,x_5]\).
This is a systematic procedure, but it leads to too many terms, which then have to be recombined to compactify the result,
a problem that is familiar from the vector field case, cf. \cite{herzog2009compute,murdock2015block,murdock2016box}.
\end{rem}
At the other side in the tensor product we have
(again, with the \(u\) powers according to the \(\mult{\mathbf{h}}\)-eigenvalues)
\bas
\ker\mult{\nM}&=&u|2\rangle \oplus u^2|5\rangle.
\eas
We now need to compute:
\bas
  \Join_0 |1,1||2\rangle&=&|2,2||2\rangle+|1,1||1\rangle, \\
\Join_1 |3,4||2\rangle&=&\frac{1}{2}|4,5||2\rangle+|3,4||1\rangle,\\
\Join_1 |1,1||5\rangle&=&|2,2||5\rangle+\frac{1}{2}|1,1||4\rangle,\\
\Join_2 |3,4||5\rangle&=&\frac{1}{2}|4,5||5\rangle+\frac{1}{2}|3,4||4\rangle,\\
\Join_1 |3,3||2\rangle&=&\frac{1}{2}|3,3||1\rangle+|4,4||2\rangle,\\
\Join_2 |3,3||5\rangle&=&\frac{1}{2}|3,3||4\rangle+\frac{1}{2}|4,4||5\rangle,\\
\Join_0 |3,3||5\rangle&=&\frac{1}{2}|5,5||5\rangle+\frac{1}{2}|4,4||4\rangle+\frac{1}{2}|3,3||3\rangle.
\eas
\begin{thm}\label{thm:reductive23}
The \((2,3)\)-reducible nilpotent has the following normal form:
\bas
F(x)&=&\begin{pmatrix} 0&1&0&0&0\\0&0&0&0&0\\0&0&0&1&0\\0&0&0&0&1\\0&0&0&0&0\end{pmatrix}
\begin{pmatrix}x_1\\x_2\\x_3\\x_4\\x_5\end{pmatrix}
+\begin{pmatrix}0\\x_1x_2 F_{1,2}^0 (x_1,x_2,x_3,x_4,x_5)\\0\\0\\x_1x_2 G_{1,2}^0 (x_1,x_2,x_3,x_4,x_5) \end{pmatrix}
\\&+&\begin{pmatrix}0\\x_1x_5 F_{1,5}^0 (x_1,x_3,x_4,x_5)\\0\\0\\x_1x_5 G_{1,5}^0 (x_1,x_3,x_4,x_5) \end{pmatrix}
+\begin{pmatrix}0\\x_2x_3 F_{2,3}^0 (x_2,x_3,x_4,x_5)\\0\\0\\x_2x_3 G_{2,3}^0 (x_2,x_3,x_4,x_5) \end{pmatrix}
+\begin{pmatrix}0\\x_3x_5 F_{3,5}^0 (x_3,x_4,x_5)\\0\\0\\x_3x_5 G_{3,5}^0 (x_3,x_4,x_5) \end{pmatrix}
\\&+&\begin{pmatrix}0\\x_1 F_{1}^0 (x_1,x_3,x_4)\\0\\0\\ x_1 G_{1}^0 (x_1,x_3,x_4)\end{pmatrix}
+\begin{pmatrix}x_1 F_{1}^1 (x_1,x_3,x_4)\\x_2 F_{1}^1 (x_2,x_4,x_5)\\0\\\frac{1}{2} x_1 G_{1}^1 (x_1,x_3,x_4)\\ x_2 G_{1}^1 (x_2,x_4,x_5)\end{pmatrix}
+\begin{pmatrix}0\\x_3 x_4 F_{4}^0 (x_3,x_4)\\0\\0\\ x_3 x_4 G_{4}^0 (x_3,x_4)\end{pmatrix}
\\&+&\begin{pmatrix}x_3 x_4 F_{4}^1 (x_3,x_4)\\\frac{1}{2}x_4 x_5 F_{4}^1 (x_4,x_5)\\0\\\frac{1}{2} x_3 x_4 G_{4}^1 (x_1,x_3,x_4)\\\frac{1}{2} x_4 x_5 G_{4}^1 (x_2,x_4,x_5)\end{pmatrix}
+\begin{pmatrix}0\\x_3 F_{3}^0(x_3)\\0\\0\\x_3 G_{3}^0 (x_3)\end{pmatrix}
+\begin{pmatrix}\frac{1}{2}x_3 F_{3}^1(x_3)\\x_4 F_{3}^1(x_4)\\0\\\frac{1}{2}x_3G_{3}^1(x_3)\\\frac{1}{2}x_4 G_{3}^1 (x_4)\end{pmatrix}
+\begin{pmatrix}0\\0\\\frac{1}{2}x_3G_{3}^2(x_3)\\\frac{1}{2}x_4 G_{3}^2 (x_4)\\\frac{1}{2}x_5G_3^2(x_5)\end{pmatrix}.
\eas
\end{thm}
\begin{rem}\label{rem:kerm}
This contradicts the \(\ker \mult{\nM}\) conjecture: no \(x_5^n e_2\) term generated.
The generating function of the kernel of \(\ker\conn{\bar{\nM}}\) is \(\frac{2}{(1-t)^5}-\frac{t}{1-t}\), the generating function of \(\ker \mult{\nM}\) is \(\frac{2}{(1-t)^5}\).
Apparently, the \(\frac{t}{1-t}\) stands for the missing powers of \(x_5\).
One might conjecture that in the \((n_1,\cdots,n_b)\)-reducible case the generating function of the kernel will be
\[
G_{\conn{\bar{\nM}}}(1,t)=\frac{b}{(1-t)^n}-\sum_{n_j<n_i} \sum_{k=1}^{n_i-n_j}\frac{t}{(1-t)^{k}},
\]
 or something along these lines.
The consequence of this last conjecture would be that the \(\ker\nM\)-conjecture would still hold when the blocks in the nilpotent are of equal size.
We will put this to the test in Section \ref{sec:22}.
\end{rem}
\begin{proof}
To check for computational errors, we apply the Cushman-Sanders test, see \cite{MR855083}:
\bas
G_{\conn{\bar{\nM}}}(u,t)&=& u+u^2 +\frac{(u+u^2)t^2}{(1-t)^5} 
+2\frac{(u+u^2)t^2}{(1-t)^4}
+\frac{(u+u^2)t^2}{(1-t)^3}
+\frac{(u^3+u^2+u+1)t}{(1-t)^3}
\\&+&\frac{(u^3+u^2+u+1)t^2}{(1-t)^2}
+\frac{(u^4+u^3+u^2+u+1)t}{(1-t)}
\eas
and it follows that
\bas
\frac{\partial}{\partial u} uG_{\conn{\bar{\nM}}}(u,t)|_{u=1}&=&5(1
+\frac{t^2}{(1-t)^5}
+\frac{2t^2}{(1-t)^4}
+\frac{t^2}{(1-t)^3}
+\frac{2t}{(1-t)^3}
+\frac{2t^2}{(1-t)^2}
+\frac{3t}{(1-t)}
)
\\&=& \frac{ 5}{(1-t)^5}.
\eas
In other words, the normal form passed the Cushman-Sanders test, its terms are linearly independent by construction and they lie in the kernel, also by construction.
This proves it is the \((2,3)\)-normal form in \(\ker\conn{\bar{\nM}}\)-style.
\end{proof}
\begin{cor}
The versal deformation is given by
\bas
\begin{pmatrix}
\alpha_1^1&1&\frac{1}{2}\alpha_3^1&{0}&0\\
\alpha_1^0&\alpha_1^1&\alpha_3^0&\alpha_3^1&{0}\\
0&0&\frac{1}{2}\beta_3^2&1&0\\
\frac{1}{2}\beta_1^1&0&\frac{1}{2}\beta_3^1&\frac{1}{2}\beta_3^2&1\\
\beta_1^0&\beta_1^1&\frac{1}{2}\beta_3^0&\frac{1}{2}\beta_3^1&\frac{1}{2}\beta_3^2\\
\end{pmatrix},
\eas
where the numbering of the \(10\) parameters \(\alpha\) and \(\beta\) corresponds to that of the \(F\) and \(G\), respectively.
\end{cor}
\subsection{The $(2,2)$-reducible case}\label{sec:22}
In Theorem \ref{thm:reductive22} we give the explicit matrix \(\nN\).
In order to construct a general formula for the elements in  \(\ker\subs{\nM}\) we proceed as follows.
First we collect the generators of the kernel: \( u x_1,u x_3,x_1x_2,x_1x_4, x_2x_3,x_3x_4\). These we multiply with \(\mathbb{R}[u x_1,x_2,u x_3,x_4]\).
Then we look for linear dependencies and if necessary, adapt the multiplying polynomial ring.
This corresponds 
to a decomposition of \(\ker\subs{\nM}\)
of type (with the \(u\) powers according to the \(\subs{\mathbf{h}}\)-eigenvalues)
\bas
\ker\subs{\nM}&=&x_1 x_2 \mathbb{R}[u x_1,x_2,u x_3,x_4]
\oplus
x_1 x_4 \mathbb{R}[u x_1,u x_3,x_4]
\\&\oplus&
x_2 x_3 \mathbb{R}[x_2,u x_3,x_4]
\oplus
x_3 x_4 \mathbb{R}[u x_3,x_4]
\\&\oplus&u(x_1 \mathbb{R}[u x_1,u x_3]
\oplus x_3\mathbb{R}[u x_3]).
\eas
We now need to compute:
\bas
  \Join_0 |1,1||2\rangle&=&|2,2||2\rangle+|1,1||1\rangle, \\
\Join_0 |3,3||2\rangle&=&|4,4||2\rangle+|3,3||1\rangle,\\
\Join_0 |1,1||4\rangle&=&|2,2||4\rangle+|1,1||3\rangle,\\
\Join_0 |3,3||4\rangle&=&|4,4||4\rangle+|3,3||3\rangle.
\eas
\begin{thm}\label{thm:reductive22}
The \((2,2)\)-reducible nilpotent has the following normal form:
\bas
F(x)&=&\begin{pmatrix} 0&1&0&0\\0&0&0&0\\0&0&0&1\\0&0&0&0\end{pmatrix}
\begin{pmatrix}x_1\\x_2\\x_3\\x_4\end{pmatrix}
+\begin{pmatrix}0\\x_1x_2 F_{1,2}^0 (x_1,x_2,x_3,x_4)\\0\\x_1x_2 G_{1,2}^0 (x_1,x_2,x_3,x_4) \end{pmatrix}
\\&+&\begin{pmatrix}0\\x_1x_4 F_{1,4}^0 (x_1,x_3,x_4)\\0\\x_1x_4 G_{1,4}^0 (x_1,x_3,x_4) \end{pmatrix}
+\begin{pmatrix}0\\x_2x_3 F_{2,3}^0 (x_2,x_3,x_4)\\0\\x_2x_3 G_{2,3}^0 (x_2,x_3,x_4) \end{pmatrix}
+\begin{pmatrix}0\\x_3x_4 F_{3,4}^0 (x_3,x_4)\\0\\x_3x_4 G_{3,4}^0 (x_3,x_4) \end{pmatrix}
\\&+&\begin{pmatrix}0\\x_1 F_{1}^0 (x_1,x_3)\\0\\ x_1 G_{1}^0 (x_1,x_3)\end{pmatrix}
+\begin{pmatrix}x_1 F_{1}^1 (x_1,x_3)\\x_2 F_{1}^1 (x_2,x_4)\\ x_1 G_{1}^1 (x_1,x_3)\\ x_2 G_{1}^1 (x_2,x_4)\end{pmatrix}
+\begin{pmatrix}0\\x_3 F_{3}^0 (x_3)\\0\\ x_3 G_{4}^0 (x_3)\end{pmatrix}
+\begin{pmatrix}x_3 F_{3}^1 (x_3)\\x_4 F_{3}^1(x_4)\\ x_3 G_{3}^1 (x_3)\\x_4G_3^1(x_4)\end{pmatrix}.
\eas
\end{thm}
\begin{proof}
To check for computational errors, we apply the Cushman-Sanders \cite{MR855083} test:
\bas
G_{\conn{\bar{\nM}}}(u,t)&=& 2u+\frac{2 u t^2}{(1-t)^4}
+2\frac{ 2ut^2}{(1-t)^3}
+\frac{2ut^2}{(1-t)^2}
+2\frac{(u^2+1)t}{(1-t)^2}
+2\frac{(u^2+1)t}{(1-t)}
\eas
and it follows that
\bas
\frac{\partial}{\partial u} uG_{\conn{\bar{\nM}}}(u,t)|_{u=1}&=&4+\frac{4  t^2}{(1-t)^4}
+2\frac{ 4t^2}{(1-t)^3}
+\frac{4t^2}{(1-t)^2}
+2\frac{4t}{(1-t)^2}
+2\frac{4t}{(1-t)}
= \frac{ 4}{(1-t)^4}.
\eas
In other words, the normal form passed the Cushman-Sanders test, its terms are linearly independent by construction and they lie in the kernel, also by construction.
This proves it is the \((2,2)\)-normal form in \(\ker\conn{\bar{\nM}}\)-style.
\end{proof}
The versal deformation is:
\[
\begin{pmatrix} \alpha_1^1&1&\alpha_3^1&0\\\alpha_1^0&\alpha_1^1&\alpha_3^0&\alpha_3^1\\\beta_1^1&0&\beta_3^1&1\\\beta_1^0&\beta_1^1&\beta_3^0&\beta_3^1\end{pmatrix},
\]
where the numbering of the \(8\) parameters \(\alpha\) and \(\beta\) corresponds to that of the \(F\) and \(G\), respectively.

The generating function of the kernel is
\bas
G_{\conn{\bar{\nM}}}(1,t)&=& 2+\frac{2  t^2}{(1-t)^4}
+2\frac{ 2t^2}{(1-t)^3}
+\frac{2t^2}{(1-t)^2}
+2\frac{2t}{(1-t)^2}
+2\frac{2t}{(1-t)}
=\frac{2}{(1-t)^4}.
\eas
This confirms the conjecture made in Remark \ref{rem:kerm}, since this is the generating function of \(\ker\nM\).
\subsection{The $(k_1,k_2)$-reducible case}\label{sec:2blocks}
We assume \(k_1\leq k_2\).
In order to construct a general formula for the elements in  \(\ker\subs{\nM}\) we proceed as follows.
First we collect the generators of the kernel: 
\bas && u^{k_1-1} x_1, u^{k_1-2} x_1 x_2,\cdots x_1 x_{k_1}, u^{ k_2-1} x_1 x_{k_1+1} , \cdots, x_1 x_n , 
\\&&
u^{k_1-2} x_2x_{k_1+1}, \cdots, x_{k_1}x_{k_1+1}, u^{k_2-1} x_{k_1+1}, u^{k_2-2} x_{k_1+1} x_{k_1+2}, \cdots,x_{k_1+1}x_n.\eas
These we multiply with \(\mathbb{R}[x_1,\cdots,x_n]\).
Then we look for linear dependencies and if necessary, adapt the multiplying polynomial ring.
This corresponds
to a decomposition of \(\ker\subs{\nM}\)
of type (with the \(u\) powers according to the \(\subs{\nH}\)-eigenvalues),
where we have to take care that the arguments \(x_j\) in the polynomial ring have an eigenvalue \(\geq\) the eigenvalue of \(x_1x_{i}\):
\bas
\ker\subs{\bar{\nM}} &=&
\bigoplus_{i=0}^{k_1-2}u^{i} x_1 x_{k_1-i} \mathbb{R}[u^{k_1-1} x_1,\cdots,u^{i} x_{k_1-i},u^{k_2-1} x_{k_1+1},\cdots ,u^{i} x_{k_1+k_2-i}]
\\&\oplus&\bigoplus_{i=0}^{k_1-2}u^{i} x_1 x_{k_1+k_2-i} \mathbb{R}[u^{k_1-1} x_1,\cdots,u^{i+1}  x_{k_1-i-1} , u^{k_2-1} x_{k_1+1},\cdots ,u^{i} x_{k_1+k_2-i}]
\\&\oplus& u^{k_1-1} x_1 \mathbb{R}[u^{k_1-1} x_1,u^{k_2-1} x_{k_1+1},\cdots ,u^{k_1-1} x_{k_2+1}]
\\&\oplus&\bigoplus_{i=0}^{k_1-2}  u^{i} x_{k_1-i}x_{k_1+1} \mathbb{R}[u^{k_1-2}x_2,\cdots,u^{i}x_{k_1-i},u^{k_2-1}x_{k_1+1},\cdots,u^{i} x_{k_1+k_2-i}]
\\&\oplus&\bigoplus_{i=0}^{k_1-3}  u^{i} x_{k_1+1} x_{k_1+k_2-i}\mathbb{R}[u^{k_1-2}x_2,\cdots,u^{i+1}x_{k_1-i-1},u^{k_2-1} x_{k_1+1},\cdots,u^{i}x_{k_1+k_2-i}]
\\&\oplus&\bigoplus_{i=k_1-2}^{k_2-2}  u^{i} x_{k_1+1} x_{k_1+k_2-i}\mathbb{R}[u^{k_2-1} x_{k_1+1},\cdots,u^{i}x_{k_1+k_2-i}]
\\&\oplus& u^{k_2-1} x_{k_1+1} \mathbb{R}[u^{k_2-1} x_{k_1+1}].
\eas
Let \(\tau=1-t\). Then
\bas
G_{\subs{\bar{\nM}}}(t,u)&=&1
+\sum_{i=0}^{k_1-2}\frac{u^{i} t^2}{(1-t)^{k_1+k_2-2i}} 
+2\sum_{i=0}^{k_1-2}\frac{u^{i} t^2}{(1-t)^{k_1+k_2-2i-1}} 
+\sum_{i=0}^{k_1-3} \frac{ u^{i} t^2}{(1-t)^{k_1+k_2-2i-2}} 
\\&+&\sum_{i=k_1-2}^{k_2-2} \frac{ u^{i} t^2}{(1-t)^{k_2-i}}
+ \frac{u^{k_1-1} t}{(1-t)^{ k_2-k_1+2}}
+ \frac{u^{k_2-1} t}{1-t}
\\&=&
1
+(1+\tau)^2(1-\tau)^2\tau^{-k_1-k_2} \sum_{i=0}^{k_1-2}(u\tau^2)^{i} 
\\&&
+(1-\tau)^2 \tau^{-k_2} \sum_{i=k_1-1}^{k_2-2} u^{i} \tau^{i}
+ (1-\tau) \tau^{k_1-k_2-2} u^{k_1-1} 
+ (1-\tau)\tau^{-1} u^{k_2-1}
\\&=&
1
+(1-\tau^2)^2\tau^{-k_1-k_2}\frac{1-(u\tau^2)^{k_1-1}}{1-u\tau^2}  
\\&&
+(1-\tau)^2 \tau^{-k_2} \frac{(u\tau)^{k_1-1}-(u\tau)^{k_2-1}}{1-u\tau}
+ (1-\tau) \tau^{k_1-k_2-2} u^{k_1-1}
+ (1-\tau)\tau^{-1} u^{k_2-1},
\eas
where the \(1\) is added to the function to make the result look better, although it is of course not in \(\ker\subs{\bar{\nM}}\).
To check for computational errors, we apply the Cushman-Sanders test to the generating function of  \(\ker\subs{\bar{\nM}}\), see \cite{MR855083}:
\bas
\frac{\partial}{\partial u} u G_{\subs{\bar{\nM}}}(u,t)|_{u=1}&=&1
+(1-\tau^2)^2\tau^{-k_1-k_2} (\frac{1 - k_1 \tau^{2k_1-2}}{1-\tau^2}+\frac{\tau^2-\tau^{2k_1} }{(1-\tau^2)^2})
+k_1 (1-\tau) \tau^{k_1-k_2-2} 
\\&&
+(1-\tau)^2 \tau^{-k_2} ( \frac{k_1 \tau^{k_1-1} -k_2 \tau^{k_2-1}}{1-\tau}+\frac{\tau^{k_1}-\tau^{k_2}}{(1-\tau)^2})
+k_2 (1-\tau)\tau^{-1} 
\\&=&1
+(1-\tau^2)\tau^{-k_1-k_2} (1 - k_1 \tau^{2k_1-2})
+\tau^{-k_1-k_2} (\tau^2-\tau^{2k_1} )
+k_1 (1-\tau) \tau^{k_1-k_2-2}
\\&&
+(1-\tau) \tau^{-k_2}  (k_1 \tau^{k_1-1} -k_2 \tau^{k_2-1})
+\tau^{-k_2} (\tau^{k_1}-\tau^{k_2})
+k_2 (1-\tau)\tau^{-1}
\\&=&
+\tau^{-k_1-k_2} 
-k_1 (1-\tau^2)\tau^{k_1-k_2-2} 
-\tau^{k_1-k_2}
+k_1 (1-\tau) \tau^{k_1-k_2-2}
\\&&
+k_1 (1-\tau)\tau^{k_1-k_2-1} 
+\tau^{k_1-k_2} 
\\&=&
\tau^{-k_1-k_2}.
\eas
as it should be.
Furthermore,
\bas
G_{\ker\subs{\bar{\nM}}}(t)&=&
G_{\subs{\bar{\nM}}}(t,1)=
+(1-\tau^2)\tau^{-k_1-k_2}(1-\tau^{2k_1-2})
\\&&
+(1-\tau) (\tau^{k_1-k_2-1}-\tau^{-1})
+ (1-\tau) \tau^{k_1-k_2-2} 
+\tau^{-1} 
\\&=&1
+\tau^{-k_1-k_2}
-\tau^{-k_1-k_2+2}.
\eas
This is as it should be: all polynomials minus those without \(x_1\) and \(x_{k_1+1}\).
\begin{lem}
\bas
\sum_{i=m_1}^{m_2} (1+i) T^i&=&
\frac{(m_1+1)T^{m_1}-m_1T^{m_1+1}-(m_2+2)T^{m_2+1}+(m_2+1)T^{m_2+2}}{(1-T)^2}.
\eas
\end{lem}
\begin{proof}
Direct computation.
\end{proof}
At the other side in the tensor product we have
(again, with the \(u\) powers according to the \(\mult{\nH}\)-eigenvalues)
\bas
\ker\mult{\nM}&=&u^{k_1-1}|k_1\rangle \oplus u^{k_2-1}|k_1+k_2\rangle.
\eas
We recall the Clebsch-Gordan formula \(u^i \otimes u^k=\sum_{j=0}^{\min(i,k)} u^{i+k-2j}\), so that the result is \(\min(i,k)+1\) when evaluated in \(u=1\).
\bas
G_{\conn{\bar{\mathfrak{m}}}}(t,1)&=&G(u,t)\otimes (u^{k_1-1}+u^{k_2-1})|_{u=1} 
\\&=&2
+2(1+\tau)^2(1-\tau)^2\tau^{-k_1-k_2} \sum_{i=0}^{k_1-2}(i+1)\tau^{2i}
\\&&
+(1-\tau)^2 \tau^{-k_2} \sum_{i=k_1-1}^{k_2-2} (k_1 + i+1) \tau^{i}
+2 (1-\tau) \tau^{k_1-k_2-2} k_1
+ (1-\tau)\tau^{-1} (k_1 +k_2)
\\&=&2
+2\tau^{-k_1-k_2}(1-k_1\tau^{2k_1-2}+(k_1-1)\tau^{2k_1})
\\&&
+k_1(1-\tau) \tau^{-k_2}  (\tau^{k_1-1} -\tau^{k_2-1})
\\&&
+\tau^{-k_2} (k_1\tau^{(k_1-1)}-(k_1-1)\tau^{k_1}-k_2\tau^{k_2-1}+(k_2-1)\tau^{k_2})
\\&&
+2 (1-\tau) \tau^{k_1-k_2-2} k_1
+ (1-\tau)\tau^{-1} (k_1 +k_2)
\\&=&
2\tau^{-k_1-k_2}
-\tau^{k_1-k_2}
+1
\\&=&
2\tau^{-k_1-k_2}
-\sum_{i=1}^{k_2-k_1}\frac{t}{\tau^{i}},
\eas
as conjectured in  Remark \ref{rem:kerm}, restricted to two blocks.

We will not write out the explicit normal form, but this is not too hard using the expressions for the transvectants.
\section{Concluding remarks}
One of the ideas one might get from this paper is that the extension of a nilpotent action into an \(\Sl\)-triple
leads to constructions that hardly need the triple in the actual computations. It is, as if the existence of the triple is enough,
and this would readily follow from the finite dimensionality of the representations and the Jacobson-Morozov theorem.
This impression is created by the fact that the \(\Sl\)-equivariant decomposition is given automatically if the nilpotent is in Jordan normal form and therefore
acts as a simple shift operator on the monomials.
It is this simplicity that allows us to give such general result on the normal form of maps with nilpotent linear part,
	much more general than is possible in the corresponding vector field case.
	\bibliographystyle{plain}

\begin{thebibliography}{10}

\bibitem{MR1735239}
Guoting Chen and Jean Della~Dora.
\newblock Normal forms for differentiable maps near a fixed point.
\newblock {\em Numer. Algorithms}, 22(2):213--230, 1999.

\bibitem{MR855083}
Richard Cushman and Jan~A Sanders.
\newblock Nilpotent normal forms and representation theory of {$\mathfrak{
  sl}(2,{\mathbb R})$}.
\newblock In {\em Multiparameter bifurcation theory ({A}rcata, {C}alif.,
  1985)}, volume~56 of {\em Contemp. Math.}, pages 31--51. Amer. Math. Soc.,
  Providence, RI, 1986.

\bibitem{cushman1988splitting}
Richard Cushman and Jan~A Sanders.
\newblock Splitting algorithm for nilpotent normal forms.
\newblock {\em Dynamics and Stability of systems}, 2(3-4):235--246, 1988.

\bibitem{cushman1988normal}
Richard Cushman, Jan~A Sanders, and Neil White.
\newblock Normal form for the (2; n)-nilpotent vector field, using invariant
  theory.
\newblock {\em Physica D: Nonlinear Phenomena}, 30(3):399--412, 1988.

\bibitem{gazor2013volume}
Majid Gazor and Fahimeh Mokhtari.
\newblock Volume-preserving normal forms of {H}opf-zero singularity.
\newblock {\em Nonlinearity}, 26(10):2809, 2013.

\bibitem{gazor2014normal}
Majid Gazor and Fahimeh Mokhtari.
\newblock Normal forms of {H}opf-zero singularity.
\newblock {\em Nonlinearity}, 28(2):311, 2014.

\bibitem{gazor2013normal}
Majid Gazor, Fahimeh Mokhtari, and Jan~A Sanders.
\newblock Normal forms for {H}opf-zero singularities with nonconservative
  nonlinear part.
\newblock {\em Journal of Differential Equations}, 254(3):1571--1581, 2013.

\bibitem{gramchev2005normal}
Todor Gramchev and Sebastian Walcher.
\newblock Normal forms of maps: formal and algebraic aspects.
\newblock {\em Acta Applicandae Mathematica}, 87(1-3):123--146, 2005.

\bibitem{herzog2009compute}
J{\"u}rgen Herzog, Marius Vladoiu, and Xinxian Zheng.
\newblock How to compute the {S}tanley depth of a monomial ideal.
\newblock {\em Journal of Algebra}, 322(9):3151--3169, 2009.

\bibitem{humphreys2012introduction}
James~E Humphreys.
\newblock {\em Introduction to {L}ie algebras and representation theory},
  volume~9.
\newblock Springer Science \& Business Media, 2012.

\bibitem{kuznetsov2004elements}
Yuri Kuznetsov.
\newblock {\em Elements of Applied Bifurcation Theory}.
\newblock Applied Mathematical Sciences. Springer New York, 2004.

\bibitem{kuznetsov2019numerical}
Yuri~A Kuznetsov and Hil~GE Meijer.
\newblock {\em Numerical Bifurcation Analysis of Maps}, volume~34.
\newblock Cambridge University Press, 2019.

\bibitem{mokhtari2019equivariant}
Fahimeh Mokhtari and Jan~A Sanders.
\newblock Equivariant decomposition of polynomial vector fields.
\newblock {\em arXiv preprint arXiv:1909.06955}, 2019.

\bibitem{mokhtari2019versal}
Fahimeh Mokhtari and Jan~A Sanders.
\newblock Versal normal form for nonsemisimple singularities.
\newblock {\em Journal of Differential Equations}, 267(5):3083--3113, 2019.

\bibitem{murdock2006normal}
James Murdock.
\newblock {\em Normal forms and unfoldings for local dynamical systems}.
\newblock Springer Science \& Business Media, 2006.

\bibitem{murdock2016box}
James Murdock.
\newblock Box products in nilpotent normal form theory: {T}he factoring method.
\newblock {\em Journal of Differential Equations}, 260(2):1010--1077, 2016.

\bibitem{murdock2015block}
James Murdock and Theodore Murdock.
\newblock Block {S}tanley decompositions {I}. {E}lementary and gnomon
  decompositions.
\newblock {\em Journal of Pure and Applied Algebra}, 219(6):2189--2205, 2015.

\bibitem{roell2020}
Ernst R\"oell.
\newblock Nilpotent normal forms for maps.
\newblock Master's thesis, Utrecht University, 2020.

\bibitem{sanders2007averaging}
Jan~A Sanders, Ferdinand Verhulst, and James~A Murdock.
\newblock {\em Averaging methods in nonlinear dynamical systems}, volume~59.
\newblock Springer, 2007.

\bibitem{wang2008further}
Duo Wang, Min Zheng, and Jianping Peng.
\newblock Further reduction of normal forms and unique normal forms of smooth
  maps.
\newblock {\em International Journal of Bifurcation and Chaos},
  18(03):803--825, 2008.

\end{thebibliography}

\end{document}